\documentclass[12pt,a4paper,reqno]{amsart}
\usepackage{amsmath,amsthm,verbatim,amssymb,amsfonts,amscd, graphicx,mathtools}
\allowdisplaybreaks
\usepackage{xcolor}
\usepackage{geometry}
\usepackage{graphicx}
\usepackage{tikz}
\usepackage[normalem]{ulem}
\usepackage[
  colorlinks=true,
  linkcolor=blue,
  citecolor=blue,
  urlcolor=blue,
  breaklinks=true
]{hyperref}
\usepackage{enumitem}
\usepackage{cancel}
\usepackage[english,capitalize]{cleveref}
\definecolor{darkblue}{rgb}{0,0,0.3}
\definecolor{urlblue}{rgb}{0,0,0.7}
\hypersetup{
	colorlinks=true,
	citecolor=blue,
	linkcolor=darkblue,
	urlcolor=urlblue,
}

\usepackage[maxbibnames=99, backend=biber, sorting=nyt, doi=false, url=false, isbn=false, style=alphabetic]{biblatex}
\newtheorem{lemma}{Lemma}
\newtheorem{question}{Question}
\newtheorem{proposition}{Proposition}
\newtheorem{theorem}{Theorem}
\newtheorem{corollary}{Corollary}
\theoremstyle{definition}

\newtheorem{remark}{Remark}
\newtheorem*{claim}{Claim}

\renewcommand{\leq}{\leqslant}
\renewcommand{\geq}{\geqslant}

\DeclareMathOperator{\Ric}{Ric}

\newcommand{\D}{\nabla}

\newcommand{\metric}[2]{\langle#1,#2\rangle}

\DeclareMathOperator{\diam}{diam}
\renewcommand{\tilde}{\widetilde}

\newcommand{\eps}{\varepsilon}

\renewcommand{\epsilon}{\varepsilon}

\title[Spectral Ricci Bounds and Mean-Convex Boundary]{Rigidity and flexibility under spectral Ricci lower bounds and mean-convex boundary}

\author{Gioacchino Antonelli}
\address[Gioacchino Antonelli]{Department of Mathematics, University of Notre Dame, Hurley Hall, 255 Hurley, Notre Dame, IN 46556, United States}
\email{gantonel@nd.edu}
\author{Yangyang Li}
\address[Yangyang Li]
{Department of Mathematics, University of Notre Dame, Hurley Hall, 255 Hurley, Notre Dame, IN 46556, United States}
\email{yyli@nd.edu}
\author{Paul Sweeney Jr}
\address[Paul Sweeney Jr]
{ Department of Mathematics, Michigan State University, 619 Red Cedar Road, C212 Wells Hall, East Lansing, MI 48824, USA}
\email{sween245@msu.edu}

\begin{document}
\begin{abstract}
    We study Riemannian manifolds $(M^n,g)$ with mean-convex boundary whose Ricci curvature is nonnegative in a spectral sense. Our first main result is a sharp spectral extension of a rigidity theorem by Kasue: we prove that under the conditions
\[
\lambda_1(-\gamma\Delta+\mathrm{Ric})\geq 0,\qquad 
H_{\partial M}\geq 0,
\]
and in the sharp range \(0\leq \gamma<4\) if \(n=2\), and \(0\leq\gamma<\frac{n-1}{n-2}\) if \(n\geq3\), a (possibly noncompact) complete manifold with disconnected boundary, with at least one compact boundary component, must split isometrically as a product $[0,L]\times \Sigma$. 

Our second main contribution is a topological rigidity result for the relative
fundamental group $\pi_1(M,\partial M)$, combined with a deep theorem of
Lawson--Michelsohn. We prove that, in dimensions $n\neq4$, any compact
manifold with boundary satisfying the two inequalities above, with at least one
of them strict, admits a metric with positive sectional curvature
and strictly mean-convex boundary, provided $\gamma\geq0$ if $n=2$, and
$0\leq\gamma\leq\frac{n-1}{n-2}$ if $n\geq3$. This range of $\gamma$ is sharp for the latter result to hold.
\end{abstract}
\maketitle
\tableofcontents
\section{Introduction}
Let $n\geq 2$, and let $(M^n,g)$ be a smooth complete manifold, possibly with boundary. In this paper, all the manifolds are assumed to be connected. We define the function $\mathrm{Ric}:M\to \mathbb R$ by
\[
\mathrm{Ric}(x):=\inf_{\substack{v\in T_xM\\ g(v,v)=1}} \mathrm{Ric}_x(v,v).
\]
Note that $\Ric\in\mathrm{Lip}_{\text{loc}}(M)$. For a constant $\gamma\geq0$, we say that \textit{$M$ satisfies $\lambda_1(-\gamma\Delta+\Ric)\geq0$}
if either of the two following equivalent conditions holds:
\begin{enumerate}[label={(\roman*)}, topsep=2pt, itemsep=1pt]
    \item for all $\varphi\in C^1_c(M)$ it holds 
    \begin{equation}\label{eqn:SpectralRicciCondition}
    \int_M\gamma|\D\varphi|^2+\Ric\cdot\varphi^2\geq0;
    \end{equation}
    \item there exist $\alpha\in (0,1)$ and $u\in C^{2,\alpha}(M)$ such that $u>0$, $-\gamma\Delta u+\Ric\cdot u\geq0$ on $M$, and $\langle \nabla u,\nu\rangle=0$ on $\partial M$. Here $\nu$ denotes the unit outward normal to the boundary. 
\end{enumerate}
The equivalence of the two conditions is readily seen on compact manifolds, and follows from the argument of \cite{Fischer-Colbrie-Schoen} on noncompact manifolds. The study of manifolds with Ricci lower bounds in the spectral sense has recently attracted substantial attention: see, e.g., \cite{APX24, AX24, CatinoMariMastroliaRoncoroni, HongWang, CS26}, and the references therein. This study has been connected to the study of stable minimal hypersurfaces, and in particular to the solution of the stable Bernstein problem in $\mathbb R^n$; we refer to \cite{ChodoshLiBernstein2, CLMS, Mazet, CabreCatinoMariMastroliaRoncoroni}.

In \cite{AX24}, the authors prove that, when $n\geq 3$, if $(M^n,g)$ is compact then
\begin{equation}\label{eqn:SharpP1}
\lambda_1\left(-\frac{n-1}{n-2}\Delta+\mathrm{Ric}\right)>0 \Rightarrow \pi_1(M)\, \text{is finite},
\end{equation}
and the constant $\frac{n-1}{n-2}$ is sharp for \eqref{eqn:SharpP1} to hold.
Moreover, in \cite{BourCarron} the authors prove that, when $n\geq 3$, if $(M^n,g)$ is compact then
\begin{equation}\label{eqn:SharpBetti}
\lambda_1\left(-\frac{n-1}{n-2}\Delta+\mathrm{Ric}\right)\geq 0 \Rightarrow b_1(M)\leq n.
\end{equation}
In \cite{AX25}, the authors prove that when $n\geq 3$ and $\gamma>\frac{n-1}{n-2}$, then the condition $\lambda_1(-\gamma\Delta+\mathrm{Ric})>0$ is stable under connected sums. Thus, one can observe that also in \eqref{eqn:SharpBetti} the constant $\frac{n-1}{n-2}$ is sharp for the property to hold. These results hint at the following natural question. 
\begin{question}\label{Question}
    Let $n\geq 2$. Define
    \[
\begin{aligned}
C(n):=\Bigl\{\gamma\geq 0:\;&\text{Any compact $M^n$ admitting a $C^\infty$ metric $g$ with $\lambda_{1}(-\gamma\Delta_g+\mathrm{Ric}_g)\geq 0$,}\\
&\text{also admits a $C^\infty$ metric $\tilde g$ such that $\mathrm{Ric}_{\tilde g}\geq 0$}\Bigr\}.
\end{aligned}
\]
Find the value of 
\[
\gamma(n):=\sup C(n).
\]
\end{question}
By using Gauss--Bonnet theorem, $\gamma(2)=+\infty$. In view of the results in \cite{AX24, AX25} explained above, we also have $\gamma(n)\leq \frac{n-1}{n-2}$ for every $n\geq 3$. Problems related to \cref{Question} with $\mathrm{Scal}$ instead of $\mathrm{Ric}$ have been considered in \cite[6.1.2]{FourLectures}, and, e.g., in \cite{LiMantoulidisFlexibility, HirschKazarasKhuriYiyueSpectralTorical}.

One of the aims of this paper is to settle a sharp boundary analogue of \cref{Question}
for compact manifolds with mean-convex boundary in dimensions $n\neq 4$;
see \cref{thm:MainThmIntro}. To the best of the authors' knowledge, \cref{Question} remains widely open in dimensions
$n\geq 5$.

\subsection{Main results}
The understanding of \cref{Question} in the setting of compact manifolds with mean-convex boundary rests on two key results. The first is a sharp spectral generalization of a rigidity theorem first proved
by Kasue \cite[Theorem B(1)]{Kasue83} (see also \cite{MR649619}): a complete manifold with $\Ric\geq0$ and disconnected
mean-convex boundary, at least one component of which is compact, must split
isometrically.  See also \cite{CrokeKleiner} for a generalization, and the recent related work \cite{CucinottaMondino}. We stress that in this paper $H_{\partial M}$ denotes the mean curvature with respect to the outward unit normal $\nu$ to the boundary $\partial M$.
\begin{theorem}\label{thmSplitting}
    Let $n\geq 2$. Let $(M^n,g)$ be a complete (possibly noncompact) manifold with nonempty boundary $\partial M$. Let $0\leq \gamma<4$ if $n=2$, or $0\leq \gamma<\frac{n-1}{n-2}$ if $n\geq 3$. Let us assume that:
    \begin{enumerate}
        \item There is $0<u\in C^{2,\alpha}(M)$ such that 
        \begin{equation}\label{eqn:TwoAssumptions}
        -\gamma\Delta u+u\mathrm{Ric}\geq 0\,\, \text{on $M$}, \qquad \text{and} \qquad H_{\partial M}+\gamma u^{-1}\langle \nabla u,\nu\rangle \geq 0\,\, \text{on $\partial M$}.
        \end{equation}
        \item $\partial M$ has at least two connected components, at least one of which is compact.
    \end{enumerate}
    Then $\mathrm{Ric}\geq 0$, $H_{\partial M}\geq 0$, and $(M^n,g)$ splits isometrically as $([0,L]\times \Sigma^{n-1},\mathrm{d}t^2+g_\Sigma)$, where $(\Sigma,g_\Sigma)$ is a compact Riemannian manifold with $\mathrm{Ric}_{g_\Sigma}\geq 0$, and $0<L<\infty$.
\end{theorem}
\begin{remark}\label{rem:Thm11}
    The following remarks are in order.
    \begin{enumerate}
        \item The fact that at least one connected component of $\partial M$ is compact is necessary in \cref{thmSplitting} when $n\geq 4$. Indeed, in $\mathbb R^4$ the minimal catenoid lies within two parallel hyperplanes. Thus, there is an open set $A$ in $\mathbb R^4$ bounded by a hyperplane and a minimal catenoid. Notice that $A$ satisfies the assumptions in  \cref{thmSplitting} (with $\gamma=0$ and $u\equiv 1$) but it does not split isometrically. 
        
        We remark that when $n\in\{2,3\}$, a manifold $M^n$ with at least two mean-convex boundary components (regardless of the compactness assumption) and $\mathrm{Ric}\geq 0$ always splits isometrically. When $n=2$ this is readily seen: one can show that there exists a geodesic line in $M$, and conclude using the splitting theorem. When $n=3$, it was recently proved in \cite[Corollary 3]{CucinottaMondino}. It would be interesting to understand if the compactness assumption in Item (2) of \cref{thmSplitting} can be dropped when $n\in\{2,3\}$. Since this question is not in the scope of this paper, we do not address it here.
        \item The thresholds $0
    \leq \gamma<4$ if $n=2$, and $0\leq \gamma<\frac{n-1}{n-2}$ if $n\geq 3$ are sharp for \cref{thmSplitting} to hold. See, respectively, the interesting hyperbolic example in \cref{prop:HyperbolicCounterexample}, and the example in \cref{eqn:Counterexample2} (inspired by \cite{AX24}).
    \item When $n=2$, if the manifold $(M^n,g)$ in \cref{thmSplitting} is compact, we obtain the same result with the weaker assumption $\gamma\geq 0$. In fact, the compactness of $M$ rules out the possible escaping at infinity of minimizing sequences in the proof of \cref{propNearbyMinimizer}, so that the same proof of \cref{propNearbyMinimizer} works with an arbitrary $\gamma$. Thus when $(M^n,g)$ is compact, arguing as in the proof of \cref{ThmRigid}, we get that, through every point $p\in M$, there is a minimizer of $E(\beta):=\int_\beta u^\gamma\mathrm{d}\ell_g$ among curves $\beta$ connecting a distinguished boundary component $\Sigma_1$ to $\partial M\setminus\Sigma_1$. Since we can assume $\gamma>0$ without loss of generality, by \cref{lemgradienvanishes} we get $|\nabla u|(p)=0$, and by the arbitrariness of $p$ we get that $u$ is constant, thereby reducing the proof to \cite{Kasue83}.
    \item \cref{thmSplitting} should be compared with the spectral splitting results of \cite{APX24} (see also \cite{CatinoMariMastroliaRoncoroni}) and \cite{HongWang}. We stress that, while for those results the sharp threshold was $0\leq \gamma<\frac{4}{n-1}$, in our case the sharp threshold is $0\leq \gamma < \frac{n-1}{n-2}$. 
    
    Roughly speaking, this is due to the following fact. In the presence of at least two boundary components, at least one of which (say $\Sigma_1$) is compact, a minimizing sequence of the energy $E(\beta):=\int_\beta u^\gamma\mathrm{d}\ell_g$ among curves connecting $\Sigma_1$ to $\partial M\setminus\Sigma_1$ cannot escape at infinity. Indeed, if it were to do so, we would find a curve $\sigma$ starting from $\Sigma_1$ with infinite $g$-length, finite $E$, such that every initial portion of $\sigma$ is $E$-minimizing among curves starting from $\Sigma_1$. Using a cut-off trick, this will contradict the stability inequality, see \cref{lemfiniteenergyfintelength}. Hence, we gain compactness of minimizers, so we can use the improved version of the stability inequality in \cref{lem:SecondVariationGeneral}, which gives information in the sharp range $0\leq \gamma<\frac{n-1}{n-2}$.
    \end{enumerate}
\end{remark}
\smallskip

It is known that if a compact manifold with boundary $(M^n,g)$ has $\mathrm{Ric}\geq 0$, and $H_{\partial M}\geq 0$, and at least one of the latter two inequalities is strict\footnote{Here and throughout, when we say that one inequality is \textit{strict},
we mean that it is strict at every point of its domain: on $M$ for the
interior inequality, and on $\partial M$ for the boundary inequality.}, then $\pi_1(M,\partial M)=0$: see \cite[Proposition 2.8]{FraserLi}, and \cite{LawsonUnknot}. Our second result is a sharp generalization of the latter property to manifolds with Ricci lower bounded in the spectral sense. We note, moreover, that in \cref{thm:Pi10} we do not require the manifold $M$ to be compact. It is enough that at least one of its boundary components is compact. 
\begin{theorem}\label{thm:Pi10}
    Let $n\geq 2$, and let $(M^n,g)$ be a complete (possibly noncompact) manifold with nonempty boundary $\partial M$. Assume that $\partial M$ has at least one compact boundary component.
    
    Let $\gamma\geq 0$. Let us assume that 
    \[
    0\leq \gamma < 4\,\, \text{if $n=2$}, \qquad \text{or} \qquad  0\leq \gamma\leq \frac{n-1}{n-2}\,\, \text{if $n\geq 3$}.
    \]
Moreover, assume there exists $0<u\in C^{2,\alpha}(M)$ such that the two inequalities in \eqref{eqn:TwoAssumptions} hold, and at least one of them is a strict inequality. 

Then $\pi_1(M,\partial M)=0$. Namely, $\partial M$ is connected, and $\iota_*:\pi_1(\partial M)\to\pi_1(M)$ is surjective.
\end{theorem}
\begin{remark}
    The following remarks are in order.
    \begin{enumerate}
    \item The fact that at least one connected component of $\partial M$ is compact is necessary in \cref{thm:Pi10} when $n\geq 4$. Indeed, in $\mathbb R^4$ 
there are smooth domains whose boundary components are
strictly mean-convex but whose mutual distance is zero.
    \item The threshold
    $0
    \leq \gamma<4$ if $n=2$, and 
    $0\leq \gamma\leq \frac{n-1}{n-2}$ if $n\geq 3$ are sharp for \cref{thm:Pi10} to hold. See, respectively, the examples in \cref{prop:HyperbolicCounterexample}, and in \cref{mainthmsharp} (inspired by \cite{AX25}).
    \item When $n=2$, if the manifold $(M^n,g)$ in \cref{thm:Pi10} is compact, and $\gamma\geq 0$, we can always conclude that $M$ is homeomorphic to a disk, and thus $\pi_1(M,\partial M)=0$. See the argument at the beginning of the proof of the forthcoming \cref{thm:MainThmIntro}.
    \end{enumerate}
\end{remark}

A deep result of Lawson--Michelsohn \cite[Theorem 1.1]{LawsonMichelsohn} asserts that if $(M^n,g)$ is a compact manifold with nonempty boundary, and $n\neq 4$, then $M^n$ admits a metric with $\mathrm{Sect}>0$ and $H_{\partial M}>0$ if and only if $\pi_1(M,\partial M)=0$ if and only if it admits a metric with $\mathrm{Ric}>0$ and $H_{\partial M}>0$. 
Hence, as a consequence of \cref{thm:Pi10} and \cite[Theorem 1.1]{LawsonMichelsohn}, we have the following answer to a boundary analogue of \cref{Question}, with the sharp constant threshold for $\gamma$.

\begin{theorem}\label{thm:MainThmIntro}
    Let $n\geq 2$ and $n\neq 4$. Let $\gamma\geq 0$ if $n=2$, or $0\leq \gamma\leq \frac{n-1}{n-2}$ if $n\geq 3$. 
    
    Let $(M^n,g)$ be a compact manifold with nonempty  boundary $\partial M$. Let $0<u\in C^{2,\alpha}(M)$ such that
    \begin{equation}\label{eqn:AssumLM}
    -\gamma\Delta u + u\mathrm{Ric}\geq 0\,\, \text{on $M$}, \qquad \text{and} \qquad H_{\partial M} + \gamma u^{-1}\langle\nabla u,\nu\rangle\geq 0\,\, \text{on $\partial M$},
    \end{equation}
    and at least one of these two inequalities is strict. Then $M^n$ admits a metric $\tilde g$ such that $\mathrm{Sect}_{\tilde g}>0$, and $H_{\partial M,\tilde g}>0$.
\end{theorem}
\begin{proof}[Proof of \cref{thm:MainThmIntro}]
Let us first deal with the case \(n=2\). Let \((M^2,g)\) be a compact surface
with nonempty boundary, and let \(\gamma\geq 0\). 
Dividing the first inequality of \eqref{eqn:AssumLM} by \(u\), we obtain
\[
\mathrm{Sect}\geq \gamma u^{-1}\Delta u .
\]
Integrating by parts gives
\[
\int_M \mathrm{Sect}
\geq
\gamma\int_M u^{-1}\Delta u
=
\gamma\int_{\partial M} u^{-1}\langle\nabla u,\nu\rangle
+
\gamma\int_M u^{-2}|\nabla u|^2 .
\]
Therefore, using the boundary assumption in \eqref{eqn:AssumLM}, and the Gauss--Bonnet theorem, we find
\[
2\pi\chi(M)=\int_M \mathrm{Sect}+\int_{\partial M}H_{\partial M}
\geq
\int_{\partial M}
\left(
H_{\partial M}+\gamma u^{-1}\langle\nabla u,\nu\rangle
\right)
+
\gamma\int_M u^{-2}|\nabla u|^2\geq 0.
\]
Moreover, if at least one of the inequalities in
\eqref{eqn:AssumLM} is strict, then we get \(\chi(M)>0\). Since \(M\) has nonempty boundary, it follows that
\(M\) is homeomorphic to the disk. In particular, \(M\) admits a metric with
\(\mathrm{Sect}>0\) and strictly mean-convex boundary \(H_{\partial M}>0\).

Let us now assume $n\geq 3$. Thus, by \cref{thm:Pi10}, we have $\pi_1(M,\partial M)=0$. Hence, by \cite[Theorem 1.1]{LawsonMichelsohn}, for $n \neq 4$, $M^n$ admits a metric with $\mathrm{Sect}>0$ and $H_{\partial M}>0$, as desired.
\end{proof}
\begin{remark}
The following remarks are in order.
\begin{enumerate}
    \item When $n\geq 3$, the threshold $0\leq \gamma\leq \frac{n-1}{n-2}$ in \cref{thm:MainThmIntro} is sharp in view of the examples constructed in \cref{mainthmsharp}. 
    
    \item The constraint $n\neq 4$ comes from the fact that it is not known whether \cite[Proposition 5.3]{LawsonMichelsohn} is true in dimension $n=4$. Notice that in dimension $n=3$, \cite[Theorem 1.1]{LawsonMichelsohn} is assuming that there are no fake $3$-cells. With the solution of the Poincaré conjecture in dimension $3$, this assumption can be removed.
    \item If $\lambda_1(-\gamma\Delta+\mathrm{Ric})\geq 0$, we have that there is $0<u\in C^{2,\alpha}(M)$ such that $-\gamma\Delta u + u\mathrm{Ric}\geq 0$ and $\langle \nabla u,\nu \rangle=0$ on $\partial M$. Thus if $\lambda_1(-\gamma\Delta+\mathrm{Ric})\geq 0$ and $H_{\partial M}\geq 0$ hold, and at least one is a strict inequality, we have \eqref{eqn:AssumLM}, where at least one is a strict inequality. Hence, \cref{thm:MainThmIntro} implies that, in the appropriate sharp range of $\gamma$, if $\lambda_1(-\gamma\Delta+\mathrm{Ric})\geq 0$ and $H_{\partial M}\geq 0$ hold, and at least one is a strict inequality, then $M$ admits a metric such that $\mathrm{Sect}>0$ and $H_{\partial M}>0$, addressing the (mean-convex) boundary analogue of \cref{Question}.
    \item When $n = 2$, the assumptions \eqref{eqn:AssumLM} (without asking that either of them be a strict inequality) imply that $M$ is either a disk, an annulus, or a M\"obius strip. This follows directly from the fact that proof of \cref{thm:MainThmIntro} in that case shows that $\chi(M)\geq 0$. 
    
    When $n\geq 3$, $n\neq 4$, $0\leq \gamma<\frac{n-1}{n-2}$, and \eqref{eqn:AssumLM} are in force (without asking either of them is a strict inequality) we have two cases. Either $\pi_1(M,\partial M)=0$, and thus $M^n$ admits a metric with $\mathrm{Sect}>0$, and $H_{\partial M}>0$; or there exists a compact Riemannian manifold $(\Sigma^{n-1},g_\Sigma)$, with $\mathrm{Ric}_{g_\Sigma}\geq 0$, and a Riemannian cover $\pi:([0,L]\times\Sigma,\mathrm{d}t^2+g_{\Sigma})\to (M^n,g)$ with $\mathrm{deg}(\pi)\leq 2$, and $0<L<\infty$. The latter rigidity comes from \cref{thmSplitting}: see the proof of \cref{thm:RigidityOrPi1Zero}. Finally, when $\gamma=\frac{n-1}{n-2}$ and \eqref{eqn:AssumLM} are in force, we cannot in general deduce any of the previous conclusions: see the examples in \cref{eqn:Counterexample2}.
    \end{enumerate}
    \end{remark}
    
\smallskip
Finally, as a corollary of \cref{thmSplitting} and \cref{thm:Pi10} we have the following sharp result, which encompasses the cases discussed above. 
\begin{corollary}\label{thm:RigidityOrPi1Zero}
    Let $n\geq 2$, $\gamma\geq 0$, and let $(M^n,g)$ be a complete (possibly noncompact) manifold with nonempty boundary $\partial M$. Assume that $\partial M$ has at least one compact boundary component.

    Let us assume that
    \begin{equation}\label{eqn:ThresholdGamma}
    0\leq \gamma < 4\,\, \text{if $n=2$}, \qquad \text{or} \qquad  0\leq \gamma\leq \frac{n-1}{n-2}\,\, \text{if $n\geq 3$};
    \end{equation}
    and that there is $0<u\in C^{2,\alpha}(M)$ such that 
        \begin{equation}\label{eqn:TwoAssumptionsAgain}
        -\gamma\Delta u+u\mathrm{Ric}\geq 0\,\, \text{on $M$}, \qquad \text{and} \qquad H_{\partial M}+\gamma u^{-1}\langle \nabla u,\nu\rangle \geq 0\,\, \text{on $\partial M$}.
        \end{equation}
    When $n\geq3$, assume in addition that either
$\gamma<\frac{n-1}{n-2}$, or at least one of the two inequalities in
\eqref{eqn:TwoAssumptionsAgain} is strict. Then at least one of the following holds:
\begin{enumerate}
    \item there exists a compact Riemannian manifold $(\Sigma^{n-1},g_\Sigma)$, with $\mathrm{Ric}_{g_\Sigma}\geq 0$, and a Riemannian cover $\pi:([0,L]\times\Sigma,\mathrm{d}t^2+g_{\Sigma})\to (M^n,g)$ with $\mathrm{deg}(\pi)\leq 2$, and $0<L<\infty$;
    \item $\pi_1(M,\partial M)=0$. Namely, $\partial M$ is connected, and $\iota_*:\pi_1(\partial M)\to\pi_1(M)$ is surjective.
    \end{enumerate}
\end{corollary}
\begin{remark}
An alternative approach to spectral Ricci lower bounds is through
$\mu$-bubbles, as in several recent works on related problems. In this paper we instead use weighted length minimizers for the conformal metric $\widetilde g=u^{2\gamma}g$. This has two advantages. First, it keeps the relevant variational objects one-dimensional,
and hence avoids additional regularity issues in high dimensions. Second, the one-dimensional framework
is best suited to prove \cref{lemfiniteenergyfintelength}, which is a key tool to prove the results in the generality of the present paper; namely, for possibly noncompact manifolds with disconnected boundary, with at least one compact boundary component. It would be interesting to understand whether our results, in this level of generality, can be
obtained by a $\mu$-bubble argument. In particular, it seems that such an
approach would require additional input to handle the noncompactness and possible singularity formation.
\end{remark}
\smallskip

\textbf{Structure of the paper.} In \cref{sec:Preliminaries} we carry out the detailed computations about the first and second variation of the functional $E(\beta):=\int_\beta u^\gamma\mathrm{d}\ell_g$. One of the main contributions of this paper is \cref{lemfiniteenergyfintelength}, which is the starting point to implement the strategy discussed in Item (4) of \cref{rem:Thm11}. In \cref{sec:Sharp} we discuss the examples that show that the thresholds on $\gamma$ in our statements are optimal. In \cref{sec:Rigid} we give the proof of \cref{thmSplitting}. The main new technical contribution is the implementation of the geodesic capturing technique in \cref{propNearbyMinimizer} in our setting. We need to carefully rule out, as explained above, the escape at infinity of minimizing sequences of the functional $E(\beta):=\int_\beta u^\gamma\mathrm{d}\ell_g$. In \cref{sec:proofs} we give the proofs of \cref{thm:Pi10} and \cref{thm:RigidityOrPi1Zero}. Finally, in the appendix (\cref{sec:app}) we show two further applications of the second variation computations in \cref{sec:Secondvariation}. First, we give a spectral generalization of \cite{MartinLiASharpComparison}, see \cref{prop:WeightedInradius}. Second, we give a slight improvement of the diameter estimate in \cite{AX24}, see \cref{rem:AX24}.
\smallskip

\textbf{Acknowledgments.} This project has received funding from the European Research Council (ERC) under
the European Union’s Horizon 2020 research and innovation programme (grant agreement
No. 947923). G.A. has been partially supported by the NSF DMS Grant No. 2550590. G.A. thanks D. Stern and G. Catino for discussions concerning
\cref{Question}, from which this work stems. The authors thank K. Xu for useful help and comments at the early stages of this project. The authors acknowledge the use of ChatGPT 5.2 for assistance with drafting routine computations, algebraic simplifications, and \TeX\,formatting. All mathematical arguments, computations, and conclusions were verified by the authors.

\section{Preliminaries}\label{sec:Preliminaries}

We will recall the first and second variation formulae for weighted geodesics. In this section, $(M^n,g)$ is a (possibly noncompact) manifold with nonempty boundary.

\subsection{First variation of weighted geodesics}\label{sec:FirstVariation}
    Let us assume $\partial M$ is not connected. Let us write $\partial M=C_1\sqcup C_2$, where $C_1,C_2$ are nonempty unions of connected components of $\partial M$. 
    Let $0<u\in C^{2,\alpha}(M)$, and
     $\sigma:[0,\ell]\to M$ be a curve parameterized by $g$-arclength. 
     
     We assume $\sigma(0)\in C_1$, $\sigma(\ell)\in C_2$, and $\sigma((0,\ell))\subset M\setminus\partial M$. Moreover, we assume that $\sigma$ is $\tilde g$-length minimizing among compact curves $\beta$ connecting $C_1$ to $C_2$, where $\tilde g=u^{2\gamma}g$. This is equivalent to saying that $\sigma$ is a minimizer of the functional $E(\beta):=\int_\beta u^\gamma \mathrm{d}\ell_g$, where $\mathrm{d}\ell_g$ is the $g$-length element, among compact curves $\beta$ connecting $C_1$ to $C_2$.  Notice that, under these assumptions, $\sigma$ must necessarily be $C^{3,\alpha}$-regular.
     \smallskip
    
    Let $e$ be a variational field along $\sigma$. Let $\epsilon$ be the variational parameter. Then, the first variation for the energy $E$ at $e$ is
\[\delta E(e)=\frac {\mathrm{d}}{\mathrm{d}\epsilon}\int_0^\ell u^\gamma\metric{\sigma'}{\sigma'}^{1/2}\,\mathrm{d}t=\int_0^\ell u^\gamma\metric{\sigma'}{\sigma'}^{1/2}\Big[\gamma\metric{u^{-1}\D u}{e}+\frac{\metric{\nabla_{\sigma'}e}{\sigma'}}{\metric{\sigma'}{\sigma'}}\Big]\,\mathrm{d}t.\]
From now on, we will tacitly assume that all the terms in the integrals are evaluated at $\sigma(t)$. At the initial time $\epsilon=0$, we get that for every variational vector field $e\perp\sigma'$ along $\sigma$ we have
\[
    0=\delta E(e)|_{\epsilon=0}=\int_0^\ell u^{\gamma}\Big[\gamma\metric{u^{-1}\D u}{e}-\metric{e}{\D_{\sigma'}\sigma'}\Big]\mathrm{d}t,
\]
and thus 
\begin{equation}\label{critical}
    \D_{\sigma'}\sigma'=\gamma u^{-1}\D^\perp u.
\end{equation}
Choosing arbitrary variations $e$ that are tangents to $\partial M$ at the endpoints of $\sigma$, we also obtain that $\sigma$ meets $\partial M$ orthogonally. Compare with \cite[Lemma 2.3]{HongWang}.

\subsection{Second variation of weighted geodesics}\label{sec:Secondvariation}
  Let us assume we are in the same setting as \cref{sec:FirstVariation}. If $e$ is a variational field along $\sigma$, the second variation is:
\begin{equation}\label{eqn:OriginalSecondVariation}
\begin{aligned}
    \delta^2E(e) &= \frac{\mathrm{d}^2}{\mathrm{d}\epsilon^2}\int_0^\ell u^\gamma\metric{\sigma'}{\sigma'}^{1/2}\,\mathrm{d}t \\
    &= \frac {\mathrm{d}}{\mathrm{d}\epsilon}\int_0^\ell\Big[\gamma u^{\gamma-1}\metric{\D u}{e}\metric{\sigma'}{\sigma'}^{1/2}+u^\gamma\frac{\metric{\nabla_{\sigma'}e}{\sigma'}}{\metric{\sigma'}{\sigma'}^{1/2}}\Big]\,\mathrm{d}t \\
    &= \int_0^\ell\Big[\gamma(\gamma-1)u^{\gamma-2}\metric{\D u}{e}^2\metric{\sigma'}{\sigma'}^{1/2}+\gamma u^{\gamma-1}\D^2 u(e,e)\metric{\sigma'}{\sigma'}^{1/2} \\
    &\qquad +\gamma u^{\gamma-1}\metric{\D u}{\D_ee}\metric{\sigma'}{\sigma'}^{1/2}+{2\gamma u^{\gamma-1}\metric{\D u}{e}\frac{\metric{\nabla_{\sigma'}e}{\sigma'}}{\metric{\sigma'}{\sigma'}^{1/2}}} \\
    &\qquad +u^\gamma\frac{\metric{\D_e\D_{\sigma'}e}{\sigma'}+\metric{\nabla_{\sigma'}e}{\nabla_{\sigma'}e}}{\metric{\sigma'}{\sigma'}^{1/2}}-u^\gamma\frac{\metric{\nabla_{\sigma'}e}{\sigma'}^2}{\metric{\sigma'}{\sigma'}^{3/2}}\Big]\,\mathrm{d}t.
\end{aligned}
\end{equation}
When $e\perp\sigma'$, notice that $\metric{\nabla_{\sigma'}e}{\sigma'}=-\metric{e}{\D_{\sigma'}\sigma'}=-\gamma u^{-1}\metric{\D u}{e}$, by \eqref{critical}.

When $e\in C_c^\infty((0,\ell);\sigma^*TM)$ and $e\perp\sigma'$, we can take the exponential (geodesic) variation, so that $\D_ee$ vanishes along $\sigma$ at $\varepsilon=0$. In this case, using also the geodesic equation \eqref{critical}, and the fact that $\sigma$ is parameterized by $g$-arclength, we get
\begin{equation}\label{eqn:Test}
\begin{aligned}
    0 &\leq \int_0^\ell u^\gamma\Big[\gamma(\gamma-1)u^{-2}\metric{\D u}{e}^2+\gamma u^{-1}\D^2 u(e,e)+\gamma u^{-1}\metric{\D u}{\D_ee} \\
    &\qquad -2\gamma^2 u^{-2}\metric{\D u}{e}^2-R(e,\sigma',\sigma',e)+\metric{\D_{\sigma'}\D_ee}{\sigma'}+|\nabla_{\sigma'}e|^2-\gamma^2u^{-2}\metric{\D u}{e}^2\Big]\,\mathrm{d}t \\
    &= \int_0^\ell u^\gamma\Big[(-2\gamma^2-\gamma)u^{-2}\metric{\D u}{e}^2+\gamma u^{-1}\D^2 u(e,e)-R(e,\sigma',\sigma',e)+|\nabla_{\sigma'}e|^2\Big]\,\mathrm{d}t.
\end{aligned}
\end{equation}
Let us now choose a relatively parallel orthonormal frame \(e_1,\ldots,e_{n-1}\) of the normal
bundle \((\sigma')^\perp\) along \(\sigma\), i.e.
\begin{equation}\label{eqn:ParallelOrthonormalFrame}
e_i\perp \sigma',\qquad \langle e_i,e_j\rangle=\delta_{ij},\qquad
(\nabla_{\sigma'}e_i)^\perp=0.
\end{equation}
Let $\varphi\in C^\infty([0,\ell])$. Using \eqref{eqn:ParallelOrthonormalFrame},
\begin{equation}\label{eqn:Trace1}
\begin{aligned}
\sum_{i=1}^{n-1} |\nabla_{\sigma'}(\varphi e_i)|^2 &= (n-1)(\varphi')^2+2\varphi\varphi'\sum_{i=1}^{n-1}\langle e_i,\nabla_{\sigma'}e_i\rangle + \varphi^2\sum_{i=1}^{n-1}|\nabla_{\sigma'}e_i|^2 \\
&=(n-1)(\varphi')^2+\varphi^2 \sum_{i=1}^{n-1}|\langle \nabla_{\sigma'}e_i,\sigma'\rangle \sigma'|^2 = (n-1)(\varphi')^2+ \gamma^2\varphi^2 u^{-2}|\nabla^\perp u|^2.
\end{aligned}
\end{equation}
Let us define $u':=(u\circ\sigma)'$ and $u''=(u\circ\sigma)''$. Notice that 
\begin{equation}\label{eqn:Trace2}
u''=\nabla^2 u(\sigma',\sigma')+\langle \nabla u,\nabla_{\sigma'}\sigma'\rangle = \nabla^2 u(\sigma',\sigma')+\gamma u^{-1}|\nabla^\perp u|^2.
\end{equation}
Let $\varphi\in C^\infty_c((0,\ell))$ and test \eqref{eqn:Test} with the normal field $\varphi e_i\in C_c^\infty((0,\ell);\sigma^*TM)$ for each $i$. Summing over $i$, we get, using \eqref{eqn:Trace1} and \eqref{eqn:Trace2}
\begin{equation}\label{eqn:second varitationsummedNoBdry}
    \begin{aligned}
    0 &\leq \int_0^\ell u^\gamma\Big[(n-1)(\varphi')^2+{\gamma^2u^{-2}|\D^\perp u|^2\varphi^2}+(-2\gamma^2-\gamma)u^{-2}|\D^\perp u|^2\varphi^2 \\
    &\qquad +\gamma u^{-1}\varphi^2\Delta u-\gamma u^{-1}\D^2 u(\sigma',\sigma')\varphi^2-\Ric(\sigma',\sigma')\varphi^2\Big]\,\mathrm{d}t \\
    &= \int_0^\ell u^\gamma\Big[(n-1)(\varphi')^2-\gamma u^{-2}|\D^\perp u|^2\varphi^2 \\
    &\qquad -\gamma u^{-1}u''\varphi^2+\varphi^2(\gamma u^{-1}\Delta u - \mathrm{Ric}(\sigma',\sigma'))\Big]\,\mathrm{d}t.
\end{aligned}
\end{equation}
Notice that the latter \eqref{eqn:second varitationsummedNoBdry} holds for every $\varphi\in C^\infty_c((0,\ell))$.

{Recall that $\sigma$ is a free boundary minimizer of $E$ connecting $C_1$ to $C_2$. Hence, for an arbitrary $\varphi\in C^\infty([0,\ell])$,  \eqref{eqn:second varitationsummedNoBdry} contains a boundary term involving the mean curvature of the boundary. In order to handle this general case, one needs to use the  second variation formula for arbitrary (not necessarily in $C_c^\infty((0,\ell);\sigma^*TM)$) variational fields $e\perp\sigma'$. Using arbitrary variations, the term $\nabla_e e$ might not vanish, and, integrating by parts, one has that
\begin{equation}\label{eqn:IBP}
\begin{aligned}
\int_0^\ell &\left(\gamma u^{\gamma-1}\langle\nabla u,\nabla_e e\rangle + u^\gamma\langle \nabla_{\sigma'}\nabla_e e,\sigma'\rangle\right)\mathrm{d}t = \\ &= -u^\gamma(\sigma(0))\mathrm{II}_{\partial M}(e(0),e(0))
-u^\gamma(\sigma(\ell))\mathrm{II}_{\partial M}(e(\ell),e(\ell)),
\end{aligned}
\end{equation}
where $\mathrm{II}_{\partial M}$ is computed with respect to the outward unit normal $\nu$: see, e.g., the last chain of equalities in the proof of \cite[Lemma 2.5]{HongWang}.

All in all, using \eqref{eqn:IBP} in the first inequality of \eqref{eqn:Test}, and arguing as in \eqref{eqn:second varitationsummedNoBdry} we have that for all $\varphi\in C^{\infty}([0,\ell])$ 
\begin{equation}\label{eqn:second varitationsummed}
    \begin{aligned}
    0 &\leq \int_0^\ell u^\gamma\Big[(n-1)(\varphi')^2-\gamma u^{-2}|\D^\perp u|^2\varphi^2-\gamma u^{-1}u''\varphi^2 \\
    & +\varphi^2(\gamma u^{-1}\Delta u - \mathrm{Ric}(\sigma',\sigma'))\Big]\,\mathrm{d}t - u^\gamma(\sigma(\ell)) \varphi^2(\ell)H_{\partial M}(\sigma(\ell))-u^\gamma(\sigma(0)) \varphi^2(0)H_{\partial M}(\sigma(0)),
\end{aligned}
\end{equation}
where $H_{\partial M}$ denotes the mean curvature computed with respect to the outward unit normal $\nu$.
Plugging $\varphi=u^{-\frac{\gamma}{n-1}}\psi$ we get the following. 
\begin{lemma}\label{lem:SecondVariationGeneral}
    Let $n\geq 2$. Let $M^n$ be a possibly noncompact manifold with disconnected boundary. Let us write $\partial M=C_1\sqcup C_2$, where $C_1,C_2$ are nonempty unions of connected components of $\partial M$.
    
    Let $0<u\in C^{2,\alpha}(M)$, and $\gamma\in\mathbb R$. Let $\sigma:[0,\ell]\to M$ be parameterized by $g$-arclength.  We assume that $\sigma(0)\in C_1$, $\sigma(\ell)\in C_2$, and $\sigma((0,\ell))\subset M\setminus\partial M$. Moreover, assume that $\sigma$ is a minimizer of the functional $E(\beta):=\int_\beta u^\gamma \mathrm{d}\ell_g$, where $\mathrm{d}\ell_g$ is the $g$-length element, among compact curves $\beta$ connecting $C_1$ to $C_2$.  
    
    Let $\nu$ be the outward unit normal to $\partial M$, $H_{\partial M}$ the mean curvature computed with respect to $\nu$, and $u_{\nu}:=\langle \nabla u,\nu\rangle$. For every $\psi\in C^\infty([0,\ell])$  we have 
    \begin{align*}
    0\leq &\int_0^\ell u^{\frac{n-3}{n-1}\gamma}\Big[
(n-1)(\psi')^2 
+\psi^2\Big[\Big(\frac{n-2}{n-1}\gamma^2-\gamma\Big)
u^{-2}(u')^2 \\
&\quad 
-\gamma u^{-2}|\nabla^\perp u|^2
+\Big(\gamma u^{-1}\Delta u
-\Ric(\sigma',\sigma')\Big)
\Big]\Big]\,\mathrm{d}t \\
&\quad
-u(\sigma(\ell))^{\frac{n-3}{n-1}\gamma}\psi(\ell)^2
\bigl(\gamma u^{-1}u_\nu+H_{\partial M}\bigr)(\sigma(\ell))
-
u(\sigma(0))^{\frac{n-3}{n-1}\gamma}\psi(0)^2
\bigl(\gamma u^{-1}u_\nu+H_{\partial M}\bigr)(\sigma(0)).
    \end{align*}
\end{lemma}
\begin{proof}
    Set
\[
a:=\frac{\gamma}{n-1},
\qquad
\varphi=u^{-a}\psi.
\]
Then
\[
\varphi'=(u^{-a}\psi)'=u^{-a}\psi'-a\,u^{-a-1}u'\psi,
\]
so
\[
(\varphi')^2
=
u^{-2a}(\psi')^2-2a\,u^{-2a-1}u'\psi\psi'
+a^2u^{-2a-2}(u')^2\psi^2.
\]
Also,
\[
u^\gamma\varphi^2=u^{\gamma-2a}\psi^2
=
u^{\frac{n-3}{n-1}\gamma}\psi^2.
\]
Hence
\[
u^\gamma(\varphi')^2
=
u^{\frac{n-3}{n-1}\gamma}(\psi')^2
-2a\,u^{\frac{n-3}{n-1}\gamma-1}u'\psi\psi'
+a^2u^{\frac{n-3}{n-1}\gamma-2}(u')^2\psi^2,
\]
and therefore
\[
(n-1)u^\gamma(\varphi')^2
=
(n-1)u^{\frac{n-3}{n-1}\gamma}(\psi')^2
-2\gamma\,u^{\frac{n-3}{n-1}\gamma-1}u'\psi\psi'
+\frac{\gamma^2}{n-1}u^{\frac{n-3}{n-1}\gamma-2}(u')^2\psi^2.
\]

Plugging this into the inequality \eqref{eqn:second varitationsummed} gives
\begin{align*}
0\leq \int_0^\ell &\Big[
(n-1)u^{\frac{n-3}{n-1}\gamma}(\psi')^2
-2\gamma\,u^{\frac{n-3}{n-1}\gamma-1}u'\psi\psi'
+\frac{\gamma^2}{n-1}u^{\frac{n-3}{n-1}\gamma-2}(u')^2\psi^2 \\
&\qquad
-\gamma u^{\frac{n-3}{n-1}\gamma-2}|\nabla^\perp u|^2\psi^2
-\gamma u^{\frac{n-3}{n-1}\gamma-1}u''\psi^2 \\
&\qquad
+\psi^2\Big(\gamma u^{\frac{n-3}{n-1}\gamma-1}\Delta u
-u^{\frac{n-3}{n-1}\gamma}\Ric(\sigma',\sigma')\Big)
\Big]\,\mathrm{d}t \\
&\quad
-u(\sigma(\ell))^{\frac{n-3}{n-1}\gamma}\psi(\ell)^2H_{\partial M}(\sigma(\ell))
-u(\sigma(0))^{\frac{n-3}{n-1}\gamma}\psi(0)^2H_{\partial M}(\sigma(0)).
\end{align*}

Now we integrate by parts the term
\[
-\gamma\int_0^\ell u^{\frac{n-3}{n-1}\gamma-1}u''\psi^2\,\mathrm{d}t.
\]
Set
\[
\alpha:=\frac{n-3}{n-1}\gamma-1.
\]
Then
\[
-\gamma\int_0^\ell u^\alpha u''\psi^2\,\mathrm{d}t
=
-\gamma\Big[u^\alpha u'\psi^2\Big]_0^\ell
+\gamma\int_0^\ell (u^\alpha\psi^2)'u'\,\mathrm{d}t.
\]
Since
\[
(u^\alpha\psi^2)'=\alpha u^{\alpha-1}u'\psi^2+2u^\alpha\psi\psi',
\]
we get
\[
-\gamma\int_0^\ell u^\alpha u''\psi^2\,\mathrm{d}t
=
-\gamma\Big[u^\alpha u'\psi^2\Big]_0^\ell
+\gamma\alpha\int_0^\ell u^{\alpha-1}(u')^2\psi^2\,\mathrm{d}t
+2\gamma\int_0^\ell u^\alpha u'\psi\psi'\,\mathrm{d}t.
\]
Substituting this back, the mixed term cancels, and one obtains
\begin{align*}
0\leq \int_0^\ell &\Big[
(n-1)u^{\frac{n-3}{n-1}\gamma}(\psi')^2 \\
&\qquad
+\Big(\frac{\gamma^2}{n-1}
+\gamma\Big(\frac{n-3}{n-1}\gamma-1\Big)\Big)
u^{\frac{n-3}{n-1}\gamma-2}(u')^2\psi^2 \\
&\qquad
-\gamma u^{\frac{n-3}{n-1}\gamma-2}|\nabla^\perp u|^2\psi^2
+\psi^2\Big(\gamma u^{\frac{n-3}{n-1}\gamma-1}\Delta u
-u^{\frac{n-3}{n-1}\gamma}\Ric(\sigma',\sigma')\Big)
\Big]\,\mathrm{d}t \\
&\quad
-\gamma\Big[u^{\frac{n-3}{n-1}\gamma-1}u'\psi^2\Big]_0^\ell \\
&\quad
-u(\sigma(\ell))^{\frac{n-3}{n-1}\gamma}\psi(\ell)^2H_{\partial M}(\sigma(\ell))
-u(\sigma(0))^{\frac{n-3}{n-1}\gamma}\psi(0)^2H_{\partial M}(\sigma(0)).
\end{align*}
Since
\[
\frac{\gamma^2}{n-1}
+\gamma\Big(\frac{n-3}{n-1}\gamma-1\Big)
=
\frac{n-2}{n-1}\gamma^2-\gamma,
\]
this may also be written as
\begin{align*}
0\leq \int_0^\ell &\Big[
(n-1)u^{\frac{n-3}{n-1}\gamma}(\psi')^2 \\
&\qquad
+\Big(\frac{n-2}{n-1}\gamma^2-\gamma\Big)
u^{\frac{n-3}{n-1}\gamma-2}(u')^2\psi^2 \\
&\qquad
-\gamma u^{\frac{n-3}{n-1}\gamma-2}|\nabla^\perp u|^2\psi^2
+\psi^2\Big(\gamma u^{\frac{n-3}{n-1}\gamma-1}\Delta u
-u^{\frac{n-3}{n-1}\gamma}\Ric(\sigma',\sigma')\Big)
\Big]\,\mathrm{d}t \\
&\quad
-\gamma\Big[u^{\frac{n-3}{n-1}\gamma-1}u'\psi^2\Big]_0^\ell \\
&\quad
-u(\sigma(\ell))^{\frac{n-3}{n-1}\gamma}\psi(\ell)^2H_{\partial M}(\sigma(\ell))
-u(\sigma(0))^{\frac{n-3}{n-1}\gamma}\psi(0)^2H_{\partial M}(\sigma(0)).
\end{align*}
Notice that, since $\sigma$ meets $\partial M$ orthogonally, and $\nu$ is the outward pointing normal, $u'(\sigma(0))=-u_\nu(\sigma(0))$, and $u'(\sigma(\ell))=+u_\nu(\sigma(\ell))$. Regrouping the boundary terms, we thus get the sought inequality in the statement.
\end{proof}
We have the following two direct consequences of \cref{lem:SecondVariationGeneral}.
\begin{lemma}\label{lemgradienvanishes}
Let \(n\geq 2\). Assume \(0<\gamma\) if \(n=2\), and
\(0<\gamma<\frac{n-1}{n-2}\) if \(n\geq 3\). Let $(M^n,g)$ be a possibly noncompact manifold with disconnected boundary, and $0<u\in C^{2,\alpha}(M)$. Let us write $\partial M=C_1\sqcup C_2$, where $C_1$ and $C_2$ are nonempty unions of connected components of $\partial M$. Let $\nu$ be the outward unit normal to the boundary.

Let $\sigma:[0,\ell]\to M$ be a minimizer of $E:\beta\mapsto \int_\beta u^\gamma\mathrm{d}\ell_g$ among compact curves $\beta$ connecting $C_1$ to $C_2$, and assume $\sigma((0,\ell))\subset M\setminus \partial M$.
Let us assume
        \begin{equation}\label{eqn:ToCheckOnlyOnSpt2}
        -\gamma\Delta u+\mathrm{Ric}\cdot u \geq 0, \quad \text{on $\sigma((0,\ell))$}, \qquad H+\gamma u^{-1}\langle \nabla u,\nu\rangle\geq 0, \quad \text{at $\sigma(0),\sigma(\ell)$}.
        \end{equation}
 Then $\nabla u = 0$ and $\gamma  \Delta u- \mathrm{Ric}\cdot u=0$ along $\sigma$; and $H+\gamma u^{-1}\langle\nabla u,\nu\rangle=0$ at $\sigma(0),\sigma(\ell)$.
\end{lemma}
\begin{proof}
    Up to reparameterizing $\sigma$ by $g$-arclength, it is enough to plug $\psi=1$ in \cref{lem:SecondVariationGeneral} and use the assumptions to get
    \[
    u'=0, \quad \nabla^\perp u=0, \quad \gamma u^{-1}\Delta u - \mathrm{Ric}(\sigma',\sigma')=\gamma u^{-1}\Delta u - \mathrm{Ric}=0,
    \]
    along $\sigma$; and $H+\gamma u^{-1}\langle\nabla u,\nu\rangle=0$ at $\sigma(0),\sigma(\ell)$, as desired.
\end{proof}
\begin{lemma}\label{lemfiniteenergyfintelength}
    Let $n\geq 2$, and let 
    \begin{equation}\label{eqn:Assumptiongammainfiniteenergy}
    \text{$0<\gamma\leq \frac{n-1}{n-2}$ if $n\geq 3$, or $0<\gamma<4$ if $n=2$}.
    \end{equation}
    Let $(M^n,g)$ be a possibly noncompact manifold with boundary and $0<u\in C^{2,\alpha}(M)$. Let $C\subset\partial M$ be a connected component of $\partial M$.  Let $\nu$ be the outward unit normal to the boundary. Let $\sigma: [0, \infty) \to M$ be a curve with infinite $g$-length, and assume that $\sigma(0)\in C$, and $\sigma((0, \infty)) \subset M \setminus \partial M$.
    Let $E:\beta\mapsto \int_\beta u^\gamma\mathrm{d}\ell_g$,  
    and suppose $E(\sigma) < +\infty$.
    
    Let us assume
        \begin{equation}\label{eqn:ToCheckOnlyOnSpt}
        -\gamma\Delta u+\mathrm{Ric}\cdot u \geq 0, \quad \text{on $\sigma((0,\infty))$}, \qquad H+\gamma u^{-1}\langle \nabla u,\nu\rangle\geq 0, \quad \text{at $\sigma(0)$}.
        \end{equation}
     When $n\geq3$, assume in addition that either
$\gamma<\frac{n-1}{n-2}$, or at least one of the two inequalities in
\eqref{eqn:ToCheckOnlyOnSpt} is strict.
     
     Then there exists some $R>0$ such that $\sigma|_{[0, R]}$ is not $E$-minimizing among compact curves connecting $C$ to $\sigma(R)$.
\end{lemma}
\begin{proof}
    Without loss of generality we can assume that $\sigma:[0,\infty)\to M$ is parameterized by $g$-arclength. 
    
    Suppose by contradiction that for any $R > 0$, $\sigma|_{[0, R]}$ is minimizing among compact curves connecting $C$ to $\sigma (R)$. Thus $\sigma$ is $C^{3,\alpha}$-regular, it meets $C$ orthogonally, and, for every $R>0$, it is stable for variations that fix $\sigma(R)$ and allow the starting point to vary in $C$.  Hence, by the computations of the second variation in \cref{sec:Secondvariation} and \cref{lem:SecondVariationGeneral}, in this case we have that for every $\psi\in C^\infty_c([0,\infty))$
    \begin{equation}\label{eqn:Secondvariationoneboundarycomponent}
    \begin{aligned}
    0\leq &\int_0^\infty u^{\frac{n-3}{n-1}\gamma}\Big[
(n-1)(\psi')^2 
+\psi^2\Big[\Big(\frac{n-2}{n-1}\gamma^2-\gamma\Big)
u^{-2}(u')^2 -\gamma u^{-2}|\nabla^\perp u|^2 \\
&\quad 
+\Big(\gamma u^{-1}\Delta u
-\Ric(\sigma',\sigma')\Big)
\Big]\Big]\,\mathrm{d}t - 
u(\sigma(0))^{\frac{n-3}{n-1}\gamma}\psi(0)^2
\bigl(\gamma u^{-1}u_\nu+H_{\partial M}\bigr)(\sigma(0)).
\end{aligned}
    \end{equation}

    Since $E(\sigma)$ is bounded, there exists a sequence $R_i \to \infty$ such that
    \[
        \varepsilon_i := \int^{R_i + 1}_{R_i} u^\gamma\,\mathrm{d}t \to 0\,.
    \]
    Let $\psi_i|_{[0,R_i]} \equiv 1$, $\psi_i|_{[R_i + 1, \infty)} \equiv 0$ be a cut-off function with $|\psi_i'| \leq 2$. Plugging $\psi_i$ in \eqref{eqn:Secondvariationoneboundarycomponent} we get
    \[
    \begin{aligned}
        0 \leq \int^\infty_0 u^{\frac{n -3}{n-1}\gamma}&[(n - 1)(\psi_i')^2 + \left(\frac{n - 2}{n - 1}\gamma^2 - \gamma\right) u^{-2}(u')^2 \psi_i^2\,  \\
        &+\psi_i^2(\gamma u^{-1}\Delta u - \mathrm{Ric})]\mathrm{d}t - u(\sigma(0))^{\frac{n-3}{n-1}\gamma}(\gamma u ^{-1}u_\nu + H_{\partial M})(\sigma(0))
    \end{aligned}
    \]
    which implies, when $n\geq 3$,
    \begin{align*}
        \int^\infty_0 &u^{\frac{n -3}{n-1}\gamma}\psi_i^2\left[\left(\gamma - \frac{n - 2}{n - 1}\gamma^2\right) u^{-2}(u')^2 + (-\gamma u^{-1}\Delta u + \mathrm{Ric})\right]\mathrm{d}t \\ 
        &+ u(\sigma(0))^{\frac{n-3}{n-1}\gamma}(\gamma u ^{-1}u_\nu + H_{\partial M})(\sigma(0))
        \leq \int^{R_i + 1}_{R_i} u^{\frac{n -3}{n-1}\gamma}(n - 1)(\psi_i')^2 \, \mathrm{d}t\\
        &\leq (n-1)\left(\int^{R_i + 1}_{R_i} u^{\gamma} \, \mathrm{d}t \right) ^{\frac{n-3}{n-1}}\left(\int^{R_i + 1}_{R_i}|\psi_i'|^{n-1} \, \mathrm{d}t \right)^{\frac{2}{n-1}}\\
        &\leq 4(n-1)\varepsilon_i^{(n-3)/(n-1)}\,.
    \end{align*}
    If $n>3$, as $i \to \infty$, the right-hand side converges to $0$, and thus, 
    \begin{equation}\label{eqn:FinalEst}
    \begin{aligned}
        \int^\infty_0 &u^{\frac{n -3}{n-1}\gamma}\left[\left(\gamma - \frac{n - 2}{n - 1}\gamma^2\right) u^{-2}(u')^2 + (-\gamma u^{-1}\Delta u + \mathrm{Ric})\right]\mathrm{d}t  \\
        &+ u(\sigma(0))^{\frac{n-3}{n-1}\gamma}(\gamma u ^{-1}u_\nu + H_{\partial M})(\sigma(0)) = 0
        \end{aligned}
    \end{equation}
    If $\gamma<\frac{n-1}{n-2}$, this implies that $u' = 0$ along $\sigma$, but $u(\sigma(0)) > 0$, contradicting $E(\sigma) < +\infty$. Thus we must have $\gamma=\frac{n-1}{n-2}$. Then, by assumption, at least one of the inequalities in \eqref{eqn:ToCheckOnlyOnSpt} is strict, which is again a contradiction with \eqref{eqn:FinalEst}. 
    
    If instead $n=3$, we choose a test function so that $\psi_i|_{[0,R_i]}\equiv 1$, $\psi_i|_{[R_i+i,\infty)}\equiv 0$ with $|\psi_i'|\leq 2i^{-1}$. The estimates above imply again \eqref{eqn:FinalEst}, from which the contradiction follows arguing verbatim as above. 

    Let us now treat the case $n=2$. Let $\varphi\in C^\infty_c([0,\infty))$. For \(n=2\), set
\[
\varphi=u^{-\gamma/2}\psi.
\]
Then
\[
\varphi'=u^{-\gamma/2}\psi'-\frac{\gamma}{2}u^{-\gamma/2-1}u'\psi,
\]
hence
\[
u^\gamma(\varphi')^2
=(\psi')^2-\gamma u^{-1}u'\psi\psi'
+\frac{\gamma^2}{4}u^{-2}(u')^2\psi^2,
\qquad
u^\gamma\varphi^2=\psi^2.
\]
Integrating by parts,
\[
-\gamma\int_0^\infty u^{-1}u''\psi^2\,\mathrm{d}t
=
-\gamma\Big[u^{-1}u'\psi^2\Big]_0^\infty
-\gamma\int_0^\infty u^{-2}(u')^2\psi^2\,\mathrm{d}t
+2\gamma\int_0^\infty u^{-1}u'\psi\psi'\,\mathrm{d}t.
\]
Substituting into the analogue of \eqref{eqn:second varitationsummed} for compactly supported variations in $[0,\infty)$, 
and using the inequalities in the assumption we obtain
\[
\begin{aligned}
0\leq \int_0^\infty \Big[
(\psi')^2+\gamma u^{-1}u'\psi\psi'
+\Big(\frac{\gamma^2}{4}-\gamma\Big)u^{-2}(u')^2\psi^2
\Big]\,\mathrm{d}t.
\end{aligned}
\]
Let $\eta>0$ to be chosen later. Now use Young's inequality
\[
ab\le \frac{\eta}{2}a^2+\frac{1}{2\eta}b^2,
\]
with
\[
a=\psi',
\qquad
b=\gamma u^{-1}u'\psi.
\]
Then
\[
\gamma u^{-1}u'\psi\psi'
\le \frac{\eta}{2}(\psi')^2+\frac{\gamma^2}{2\eta}u^{-2}(u')^2\psi^2.
\]
Hence
\begin{equation}\label{eqn:Tesn=2}
\int_0^\infty \Big(-\frac{\gamma^2}{4}+\gamma-\frac{\gamma^2}{2\eta}\Big)u^{-2}(u')^2\psi^2 \mathrm{d}t \leq \int_0^\infty
\Big(1+\frac{\eta}{2}\Big)(\psi')^2\mathrm{d}t.
\end{equation}
Since $\gamma<4$, we can choose $\eta>0$ large enough so that $-\frac{\gamma^2}{4}+\gamma-\frac{\gamma^2}{2\eta}>0$. As done before, we test \eqref{eqn:Tesn=2} with $\psi_i|_{[0,R_i]}\equiv 1$, $\psi_i|_{[R_i+i,\infty)}\equiv 0$ with $|\psi_i'|\leq 2i^{-1}$, where $R_i\to\infty$. As $i\to \infty$ we find $u'=0$ along $\sigma$, but $u(\sigma(0))>0$, contradicting $E(\sigma)<+\infty$.
\end{proof}

\section{Sharpness of the range}\label{sec:Sharp}
We present here the examples that show the sharpness of our results.
\begin{proposition}\label{mainthmsharp}
     Let $n\geq 3$. For every $\gamma>\frac{n-1}{n-2}$ there is a compact Riemannian manifold with boundary $(M^n,g)$ such that:
     \begin{enumerate}
         \item $\lambda_1(-\gamma\Delta+\mathrm{Ric})> 0$;
         \item $\mathrm{Ric}\not\geq 0$;
         \item $\mathrm{II}_{\partial M}> 0$;
         \item $\pi_1(M,\partial M)\neq 0$.
     \end{enumerate}
     In particular, by \cite[Proposition 2.8]{FraserLi}, $(M^n,g)$ does not admit a metric with $\mathrm{Ric}>0$ and $H_{\partial M}>0$.
    \end{proposition}
        \begin{proof}
        In this proof, we will take $M^n=(\mathbb S^1\times \mathbb S^{n-1})\# \mathbb D^n$, and we will construct a metric $g$ satisfying the requirements in the statement. Let $\gamma>\frac{n-1}{n-2}$. Take $M_1$ to be $\mathbb S^1\times \mathbb S^{n-1}$ endowed with a Riemannian metric satisfying $\lambda_1(-\gamma\Delta+\mathrm{Ric})>0$ as in \cite[Remark 4]{AX24}. This metric exists because $\gamma>\frac{n-1}{n-2}$, and can be taken so that the lowest eigenvalue of the Ricci tensor is $<0$, see \cite[Remark 4]{AX24}. Let $x\in M_1$. Thus, by \cite[Theorem 2.3]{CatinoMariMastroliaRoncoroni} there is $\eta>0$ such that there is $u_1>0$ with $-\gamma\Delta u_1 +u_1\mathrm{Ric}=\eta u_1 +\delta_x$ on $M_1$.
        
        Take $M_2$ to be a round spherical cap with convex boundary. We clearly have $\lambda_1(-\gamma\Delta+\mathrm{Ric})>0$. Let $y\in \mathrm{int}(M_2)$. Up to possibly lowering $\eta$, we have that there is $u_2>0$ such that $-\gamma\Delta u_2+u_2\mathrm{Ric}=\eta u_2+\delta_y$ on $M_2\setminus \partial M_2$, and $\langle \nabla u_2,\nu\rangle =0$ on $\partial M_2$, where $\nu$ is the outward unit normal on $\partial M_2$.
        
        Take the disjoint union $M=M_1\sqcup M_2$, and $u:=u_i$ on $M_i$. Thus $-\gamma\Delta u+u\mathrm{Ric}=\eta u + (\delta_x+\delta_y)$ on $M$. We can thus perturb the metric around $x,y$ in $M$, and the function $u$ locally around $x,y$ as in \cite{AX25}. Namely, by \cite[Proposition 1]{AX25}, there exists a Riemannian metric $\tilde g$ on $M_1\# M_2$ that coincides with the one on $M_1\sqcup M_2$ on $M\setminus(B_\varepsilon(x)\cup B_\varepsilon(y))$, with $\varepsilon\ll 1$; and a smooth function $\tilde u$ that is equal to $u$ on $M\setminus(B_\varepsilon(x)\cup B_\varepsilon(y))$ (and thus $\langle\nabla u,\nu\rangle=0$ on $\partial M$ when $\varepsilon\ll 1$) such that $-\gamma\widetilde\Delta \tilde u + \widetilde u\widetilde{\mathrm{Ric}}\geq \frac{\eta\widetilde u}{2}$. The latter conditions imply $\lambda_1(-\gamma\Delta+\mathrm{Ric})>0$ on $\tilde M$. 
        
        Moreover, from the construction, we have $\mathrm{II}_{\partial\tilde M}>0$, because the metric is not perturbed in a small neighborhood of the boundary if $\varepsilon\ll 1$, and $\pi_1(\partial\tilde M)\to \pi_1(\tilde M)$ is not surjective because $\tilde M$ is not simply connected, while $\partial\tilde M$ is simply connected. Thus also $\pi_1(\tilde M,\partial\tilde M)\neq 0$, as desired. Finally, since the metric is not changed on $M\setminus(B_\varepsilon(x)\cup B_\varepsilon(y))$, the construction can be made so that the lowest eigenvalue of $\mathrm{Ric}$ is $<0$, because this was the case for $M_1$.
    \end{proof}

\begin{proposition}\label{eqn:Counterexample2}
For every \(n\geq 3\), there exists a compact Riemannian manifold with boundary \((M^n,g)\) such that
\begin{equation}\label{eqn:Lambda1Counter}
\lambda_1\left(-\frac{n-1}{n-2}\Delta+\Ric\right) \geq 0,
\end{equation}
\(H_{\partial M}=0\), and neither \(\pi_1(M,\partial M)=0\) nor there exists a Riemannian cover $([0,L]\times\Sigma,\mathrm{d}t^2+g_\Sigma)\to (M,g)$, where $\mathrm{Ric}_{g_\Sigma}\geq 0$, and $0<L<\infty$.
\end{proposition}

\begin{proof}
The example we construct is topologically $\mathbb S^{n-1}\times [0,1]$. It is essentially the one constructed in \cite[Remark 4]{AX24}, see also \cite{BourCarron}. Let $\mathbb S^1\cong [0,2\pi]$ with the endpoints identified. Let \(f:[0,2\pi]\to (0,\infty)\) be a smooth positive nonconstant periodic function such that there are \(a<b\in [0,2\pi]\) with \(a\neq b\) satisfying
\[
f'(a)=f'(b)=0.
\] 
Assume moreover there exists $c\in (a,b)$ such that 
\begin{equation}\label{eqn:RicciViolated}
f''(c)>0.
\end{equation}
Choosing $f(t):=\sin(t)+2$, $a=\frac{\pi}{2}$, $b=\frac{3\pi}{2}$, and $c=\frac{5\pi}{4}$ is enough. Let \(N=\mathbb S^1\times \mathbb S^{n-1}\), and equip it with a warped product metric
\[
g=\mathrm{d}t^2+\varepsilon^2 f(t)^2g_{\mathbb S^{n-1}},
\]
where $\varepsilon>0$ is small enough so that $\mathrm{Ric}_g(\partial_t,\partial_t)=\mathrm{Ric}_g$. Let $
u(t):=f(t)^{2-n}$. 
The computations in \cite[Remark 4]{AX24} show that
\begin{equation}\label{eqn:SpectralRicciVerified}
-\frac{n-1}{n-2}\Delta u+\Ric\cdot u = 0,
\end{equation}
everywhere on $N$. Let
\[
M:=\{\,a\leq t\leq b\,\}\subset \mathbb S^1\times \mathbb S^{n-1}.
\]
Since
\[
H_{\{t=t_0\}}=\pm (n-1)\frac{f'(t_0)}{f(t_0)},
\]
where the sign depends on the choice of the unit normal on the boundary, and \(f'(a)=f'(b)=0\), we have
\[
H_{\partial M}=0
\qquad\text{on }\partial M.
\]

Moreover, since \(u=u(t)\), we also have
\begin{equation}\label{eqn:partialnu}
\partial_\nu u=0
\qquad\text{on }\partial M,
\end{equation}
because \(u'(t)=(2-n)f(t)^{1-n}f'(t)\).

Hence, by \eqref{eqn:SpectralRicciVerified}, and \eqref{eqn:partialnu}, we have \eqref{eqn:Lambda1Counter}. Since \(M\) has two boundary components, \(\pi_1(M,\partial M)\neq 0\).
Finally, by \eqref{eqn:RicciViolated} we have that for every $p\in\mathbb S^{n-1}$, $\mathrm{Ric}_{(c,p)}(\partial _t,\partial_t)=-(n-1)\frac{f''(c)}{f(c)}<0$, and hence $\mathrm{Ric}$ is not nonnegative on $M$. Thus, there is no Riemannian cover $([0,L]\times\Sigma,\mathrm{d}t^2+g_\Sigma)\to (M,g)$, where $\mathrm{Ric}_{g_\Sigma}\geq 0$, otherwise we would have $\mathrm{Ric}_{(M,g)}\geq 0$.
\end{proof}
\begin{proposition}\label{prop:HyperbolicCounterexample}
    There exists a complete noncompact surface $(M^2,g)$ with exactly two compact geodesic boundary components such that 
    \[
    \lambda_1(-4\Delta+\mathrm{Sect})\geq 0,
    \]
    but $M$ does not split isometrically as $([0,L]\times\Sigma,\mathrm{d}t^2+g_\Sigma)$, for a compact $\Sigma$, and $0<L<\infty$.
\end{proposition}
\begin{proof}
The example we construct is a (symmetric) complete hyperbolic surface of genus one with one funnel end (topologically, a one-holed torus) cut open along one geodesic. Let $\mathbb D$ be the Poincaré disk, equipped with the hyperbolic metric
\[
g_{\mathbb H}=\frac{4|\mathrm{d}z|^2}{(1-|z|^2)^2}.
\]
For $\xi\in \partial\mathbb D$, define the Poisson kernel on $\mathbb D$
\[
P(z,\xi):=\frac{1-|z|^2}{|z-\xi|^2}.
\]
A standard computation gives
\[
\Delta P(\cdot,\xi)^s
=
s(s-1)P(\cdot,\xi)^s.
\]
In particular, for $s=1/2$,
\begin{equation}\label{eqn:CompPoisson}
-4\Delta P(\cdot,\xi)^{1/2}
-
P(\cdot,\xi)^{1/2}
=0.
\end{equation}

We now construct a rank-two Schottky group with a reflection symmetry. We refer the reader to \cref{fig:schottky-domain}.
Let $a_0\in\partial \mathbb D$. Let
$\mathcal G\subset\mathbb D$ be the geodesic joining $a_0$ to $-a_0$, and let
$\rho$ be reflection across $\mathcal G$. Notice $\rho=\rho^{-1}$.
Choose a hyperbolic isometry $A$ whose axis is $\mathcal G$, and which translates along $\mathcal G$. Then
\[
\rho A\rho^{-1}=A.
\]
Next, choose a geodesic $\mathcal A_B$ whose endpoints are $p$ and
$\rho(p)$, where $p\in\partial\mathbb D\setminus\{a_0,-a_0\}$, and choose a
hyperbolic translation $B$ along $\mathcal A_B$. Since $\rho$ exchanges
the two endpoints of $\mathcal A_B$, conjugation by $\rho$ preserves the
axis of $B$ but reverses its translation direction. Therefore,
\[
\rho B\rho^{-1}=B^{-1}.
\]

Let us define the group
\[
\Gamma:=\langle A,B\rangle.
\]
Recall that the critical exponent of $\Gamma$ is
\begin{equation}\label{eqn:deltaLambda}
\delta(\Gamma)
:=
\inf\left\{
s>0:
\sum_{\gamma\in\Gamma}
e^{-s d_{\mathbb H}(o,\gamma o)}
<+\infty
\right\},
\end{equation}
where $o\in\mathbb D$ is any fixed base point. This definition is independent
of $o$. More generally, if $x,y\in\mathbb D$ are fixed, then, by the triangle inequality, the value of $\delta(\Gamma)$ stays the same if we substitute the series
\[
\sum_{\gamma\in\Gamma} e^{-s d_{\mathbb H}(o,\gamma o)}\quad \text{with}\quad \sum_{\gamma\in\Gamma} e^{-s d_{\mathbb H}(x,\gamma y)},
\]
in \eqref{eqn:deltaLambda}.
Let us recall that the translation length of $A$ is $\tau(A):=\inf_{x\in \mathbb D}d_{\mathbb H}(x,Ax)$. This infimum is positive, and attained along the axis $\mathcal{G}$ of $A$. Likewise for $B$.

\begin{claim}
    If the translation lengths of $A$ and $B$ are sufficiently large, we have that $\Gamma$ is a torsion-free rank-two Schottky group, and $\delta(\Gamma)<\frac{1}{2}$.
\end{claim}

\begin{proof}[Proof of the claim]
    This is certainly known to experts in hyperbolic geometry, and in particular to those familiar with the Patterson-Sullivan formulation of the critical exponent in
\cite{PattersonExponent,PattersonLimit,SullivanDensity}. We include an argument for the sake of readability, leveraging the notes in \cite{QuintPS}. 

Fix $o\in\mathbb D$. By \cite[Lemma 4.6]{QuintPS}, after replacing $A$ and $B$
by sufficiently large powers $A^N$ and $B^N$, the group $\Gamma_N:=\langle A^N,B^N\rangle$
is free and discrete. Moreover, the same lemma gives a linear lower bound
for $d(o,\gamma o)$ in terms of the reduced word length of $\gamma$. Applying the estimate in \cite[Lemma~4.6]{QuintPS}, and using that $\tau(A^{\pm N})=N\tau(A)$ (and likewise for $B$), we
obtain constants $c>0$ and $C\geq 0$, independent of $N$, such that every
nontrivial reduced word $\gamma\in\Gamma_N$ of word length $m$ satisfies
\[
d_{\mathbb H}(o,\gamma o)\geq (cN-C)m .
\]
We now estimate the Poincaré series. The number of reduced words of length
$m$ in the free group on two generators is $4\cdot 3^{m-1}$. Hence, for
every $s>0$,
\[
\sum_{\gamma\in\Gamma_N} e^{-s d_{\mathbb H}(o,\gamma o)}
\leq
1+\sum_{m\geq1}4\cdot 3^{m-1}e^{-s(cN-C)m}.
\]
The series on the right-hand side converges whenever
\[
3e^{-s(cN-C)}<1,
\]
or equivalently whenever
\[
s>\frac{\log 3}{cN-C}.
\]
Therefore the critical exponent of $\Gamma_N$ satisfies
\[
\delta(\Gamma_N)\leq \frac{\log 3}{cN-C}.
\]
Hence, for $N$ sufficiently large, we can arrange for $
\delta(\Gamma_N)<\frac12$.
\end{proof}

Replacing $A$ and $B$ with sufficiently large powers, as in the proof of the claim above, and relabeling, we can assume $\delta(\Gamma)<1/2$.
The quotient
\[
\Sigma:=\mathbb D/\Gamma
\]
is a complete hyperbolic surface of genus one with one funnel end (topologically, a one-holed torus). With our choices, $\mathcal{G}$
projects to a simple nonseparating closed geodesic on $\Sigma$. We denote
its projection again by $\mathcal{G}$.

We denote by $\Lambda(\Gamma)$ the limit set, namely the set of accumulation
points in $\partial\mathbb D$ of one, equivalently every, orbit $\Gamma x$;
see \cite[\S4.3]{QuintPS}. Its complement $
\Omega(\Gamma):=\partial\mathbb D\setminus\Lambda(\Gamma)$
is called domain of discontinuity.
The relations
\[
\rho A\rho^{-1}=A,
\qquad
\rho B\rho^{-1}=B^{-1},
\]
imply that
\[
\rho\Gamma\rho^{-1}=\Gamma.
\]
Thus $\rho$ preserves the limit set $\Lambda(\Gamma)$ and also the domain
of discontinuity $
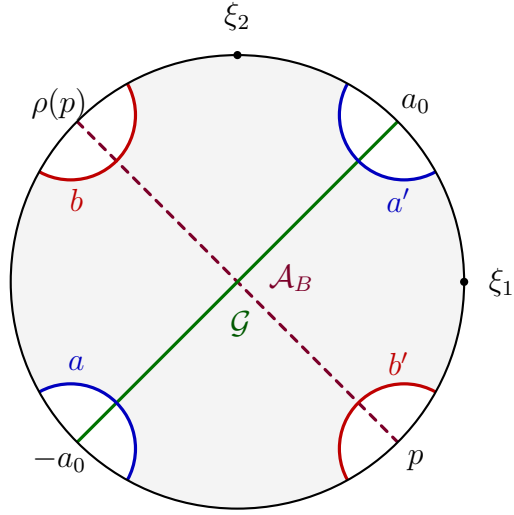
\Omega(\Gamma)$.

\begin{figure}[htbp]
\centering
\begin{tikzpicture}[scale=3.0,line cap=round,line join=round]
  \def\r{0.2857}
  \coordinate (O) at (0,0);
  \coordinate (Ca) at (-0.7354,0.7354);
  \coordinate (Cap) at (0.7354,-0.7354);
  \coordinate (Cb) at (-0.7354,-0.7354);
  \coordinate (Cbp) at (0.7354,0.7354);

  \begin{scope}
    \clip (O) circle (1);
    \fill[gray!9] (O) circle (1);
    \fill[white] (Ca) circle (\r);
    \fill[white] (Cap) circle (\r);
    \fill[white] (Cb) circle (\r);
    \fill[white] (Cbp) circle (\r);
    \draw[purple!65!black,dashed,very thick] (-0.7071,0.7071) -- (0.7071,-0.7071);
    \draw[green!45!black,very thick] (-0.7071,-0.7071) -- (0.7071,0.7071);
    \draw[red!75!black,very thick] (Ca) circle (\r);
    \draw[red!75!black,very thick] (Cap) circle (\r);
    \draw[blue!75!black,very thick] (Cb) circle (\r);
    \draw[blue!75!black,very thick] (Cbp) circle (\r);
  \end{scope}

  \draw[black,thick] (O) circle (1);

  \node[red!75!black] at (-0.71,0.36) {\(b\)};
  \node[black!75!black] at (0.7854,0.7854) {\(a_0\)};
  \node[black!75!black] at (-0.7854,-0.7854) {\(-a_0\)};
  \node[black!75!black] at (-0.7854,0.7854) {\(\rho(p)\)};
  \node[black!75!black] at (0.7854,-0.7854) {\(p\)};
  \node[red!75!black] at (0.71,-0.36) {\(b'\)};
  \node[blue!75!black] at (-0.71,-0.36) {\(a\)};
  \node[blue!75!black] at (0.71,0.36) {\(a'\)};
  \node[green!35!black,below right] at (-0.08,-0.08) {\(\mathcal{G}\)};
  \node[purple!65!black,above right] at (0.09,-0.09) {\(\mathcal{A}_B\)};

  \fill[black] (0:1) circle (0.018);
  \node[right] at (0:1.06) {\(\xi_1\)};
  \fill[black] (90:1) circle (0.018);
  \node[above] at (90:1.06) {\(\xi_2\)};
  \node at (0,-1.55) {\(\text{Side pairings: } a\xrightarrow{A}a',\quad b\xrightarrow{B}b'\)};
\end{tikzpicture}
\caption{A symmetric classical Schottky fundamental domain (in grey)
for $\Sigma=\mathbb D/\Gamma$. The red (geodesic) sides are exchanged by
$\rho$ and paired by $B$; the blue (geodesic) sides are fixed setwise by $\rho$
and paired by $A$. The green diameter $\mathcal{G}$ is used for the cut.}
\label{fig:schottky-domain}
\end{figure}

Since the endpoints $a_0$ and $-a_0$ of $\mathcal G$ are the fixed points of
$A$, they belong to $\Lambda(\Gamma)$. Choose
$
\eta\in \Omega(\Gamma)
$
with $\eta\notin\{a_0,-a_0\}$ 
and define
\[
\xi_1:=\eta,
\qquad
\xi_2:=\rho(\eta).
\]
Since $\rho$ preserves $\Omega(\Gamma)$, we have $
\xi_1,\xi_2\in\Omega(\Gamma)$.
Define
\[
f(z):=P(z,\xi_1)^{1/2}+P(z,\xi_2)^{1/2}.
\]
Then $f>0$, and by \eqref{eqn:CompPoisson},
\[
-4\Delta f-f=0
\qquad
\text{on } \mathbb D.
\]
Moreover, since $\rho$ exchanges $\xi_1$ and $\xi_2$, and $P(\rho z,\rho\xi)=P(z,\xi)$, we have
\[
f\circ\rho=f.
\]

We now define the Poincaré series
\begin{equation}\label{eqn:PoincareSeries}
F(z):=\sum_{\gamma\in\Gamma} f(\gamma z).
\end{equation}
\begin{claim}
    $F(z)$ converges locally uniformly on $\mathbb D$. 
\end{claim}
\begin{proof}[Proof of claim]
We provide a proof for the sake of readability. Let
$K\Subset\mathbb D$, and fix
$\xi\in\Omega(\Gamma)$. We first prove that the orbit $\Gamma K$ stays a
positive Euclidean distance away from $\xi$. Indeed, suppose not. Then there
exist $\gamma_j\in\Gamma$ and $z_j\in K$ such that
\[
\gamma_j z_j\to \xi.
\]
After passing to a subsequence, $z_j\to z_\infty\in K$. If the sequence
$\gamma_j$ were eventually constant (up to subsequences), then $\gamma_j z_j$ would converge to
a point of $\mathbb D$, not to the boundary point $\xi$. Thus
$\gamma_j\to\infty$ in $\Gamma$. Since accumulation points on
$\partial\mathbb D$ of sequences $\gamma_j z_\infty$, with
$\gamma_j\to\infty$, belong to the limit set $\Lambda(\Gamma)$, and since
\[
d_{\mathbb H}(\gamma_j z_j,\gamma_j z_\infty)
=
d_{\mathbb H}(z_j,z_\infty)\to 0,
\]
the points $\gamma_j z_j$ and $\gamma_j z_\infty$ converge to the same
boundary point. Therefore $\xi\in\Lambda(\Gamma)$, contradicting
$\xi\in\Omega(\Gamma)$. Hence there exists $c:=c(K,\xi)>0$ such that
\[
|\gamma z-\xi|\geq c
\qquad
\text{for all } z\in K,\ \gamma\in\Gamma.
\]

It follows that, for $C:=C(K,\xi)$
\[
P(\gamma z,\xi)^{1/2}
=
\frac{(1-|\gamma z|^2)^{1/2}}{|\gamma z-\xi|}
\leq
C(1-|\gamma z|^2)^{1/2}.
\]
In the Poincaré disk,
\[
e^{-d_{\mathbb H}(0,w)}
=
\frac{1-|w|}{1+|w|} \Rightarrow
1-|w|^2
=
(1+|w|)^2 e^{-d_{\mathbb H}(0,w)}
\leq
4e^{-d_{\mathbb H}(0,w)}.
\]
Applying this with $w=\gamma z$, we get
\[
P(\gamma z,\xi)^{1/2}
\leq
C e^{-\frac12 d_{\mathbb H}(0,\gamma z)}.
\]
Let $z_0\in \mathbb D$. Since $K$ has bounded hyperbolic diameter, using the triangle inequality, we thus conclude
\begin{equation}\label{eqn:KeyEstimate}
P(\gamma z,\xi)^{1/2}
\leq
C(K,\xi,z_0) e^{-\frac12 d_{\mathbb H}(0,\gamma z_0)}
\end{equation}
for all $z\in K$ and all $\gamma\in\Gamma$.

Since $\delta(\Gamma)<1/2$, we have
\[
\sum_{\gamma\in\Gamma}
e^{-\frac12 d_{\mathbb H}(0,\gamma z_0)}
<+\infty.
\]
Therefore, for every $\xi\in\Omega(\Gamma)$, by \eqref{eqn:KeyEstimate} we have
\[
\sum_{\gamma\in\Gamma}
\sup_{z\in K}P(\gamma z,\xi)^{1/2}
<+\infty.
\]
Applying this estimate to both $\xi=\xi_1$ and $\xi=\xi_2$, we conclude
that
\[
\sum_{\gamma\in\Gamma}
\sup_{z\in K} f(\gamma z)
<+\infty.
\]
Thus $F$ converges uniformly on $K$. Since $K\Subset\mathbb D$ was
arbitrary, $F$ converges locally uniformly on $\mathbb D$.
\end{proof}

Let $F_N$ be any sequence of finite partial sums of the Poincaré series defined in \eqref{eqn:PoincareSeries}. Since each summand
$f\circ\gamma$ satisfies
\[
-4\Delta(f\circ\gamma)-f\circ\gamma=0,
\]
we have $
-4\Delta F_N-F_N=0$.
Passing to the locally uniform limit, $F$ is a distributional solution of
\begin{equation}\label{eqn:LaplacianF}
-4\Delta F-F=0.
\end{equation}
By elliptic regularity $
F\in C^\infty(\mathbb D)$,
and local elliptic estimates imply that the convergence is smooth on compact
subsets. Moreover, $
F>0$.

Now, notice that the function $F$ is $\Gamma$-invariant by construction. 
Therefore $F$ descends to a smooth positive function, still denoted by $F$, on $\Sigma$,
and satisfies
$-4\Delta F-F=0$ on $\Sigma$.

We now check the reflection symmetry. Since $f\circ\rho=f$ and
$\rho\Gamma\rho^{-1}=\Gamma$, we have
\[
\begin{aligned}
F(\rho z) &=
\sum_{\gamma\in\Gamma} f(\gamma\rho z) =
\sum_{\gamma\in\Gamma} f\bigl(\rho(\rho\gamma\rho^{-1})z\bigr) \\
&=
\sum_{\gamma\in\Gamma} f\bigl((\rho\gamma\rho^{-1})z\bigr) =
F(z).
\end{aligned}
\]
Thus, $F$ is even with respect to the reflection $\rho$. The fixed-point set
of $\rho$ is the geodesic $\mathcal G$. Therefore
\begin{equation}\label{eqn:PartialnuF}
\partial_\nu F=0
\qquad
\text{along } \mathcal G,
\end{equation}
in $\Sigma$, where $\nu$ is either unit normal to $\mathcal G$.

Because $\mathcal G$ is a simple nonseparating closed geodesic on $\Sigma$, cutting $\Sigma$ open along $\mathcal G$
produces a connected complete hyperbolic surface $M$ with exactly two compact totally
geodesic boundary components (and one funnel end). The two boundary components
are the two sides of the same geodesic $\mathcal G$. By \eqref{eqn:LaplacianF} and \eqref{eqn:PartialnuF}, the restriction of $F$ to $M$ satisfies
\[
-4\Delta F-F=0
\qquad
\text{on } M,
\]
and
\[
\partial_\nu F=0
\qquad
\text{on } \partial M.
\]

Since $\operatorname{Sect}\equiv -1$, this is equivalently
\[
(-4\Delta+\operatorname{Sect})F=0
\qquad
\text{on } M,
\]
with Neumann boundary condition. Thus we conclude that $\lambda_1(-4\Delta+\mathrm{Sect})\geq 0$, as desired. Finally, by construction, $M$ does not split isometrically as $[0,L]\times\Sigma$ with a compact $\Sigma$, as desired.
\end{proof}
\section{The rigidity statement: Proof of \cref{thmSplitting}}\label{sec:Rigid}

    This section is aimed at proving \cref{thmSplitting}. In case $\gamma=0$, the theorem reduces to the classical Kasue's rigidity result \cite[Theorem B(1)]{Kasue83}. Therefore, it is enough to prove the following. We expect that the same approach gives spectral generalizations of the results in \cite{CrokeKleiner}. We will not discuss this here, since this lies outside the scope of the present paper.
    \begin{theorem}\label{ThmRigid}
Let $n\geq 2$. Let $(M^n,g)$ be a possibly noncompact manifold with disconnected boundary. Assume at least one boundary component is compact. Let $\nu$ be the outward unit normal to the boundary. Let $0<\gamma<\frac{n-1}{n-2}$ if $n\geq 3$, or $0< \gamma<4$ if $n=2$. Assume there is $u\in C^{2,\alpha}(M)$, $u>0$, such that 
        \[
        -\gamma\Delta u+\mathrm{Ric}\cdot u \geq 0, \quad \text{on $M$}, \qquad H_{\partial M}+\gamma u^{-1}\langle \nabla u,\nu\rangle\geq  0, \quad \text{on $\partial M$},
        \]
        where $H_{\partial M}$ is computed with respect to the outward unit normal $\nu$ to $\partial M$.
        
        Then $u$ is constant and $M$ is isometric to $([0,L]\times \Sigma,\mathrm{d}t^2+g_\Sigma)$, where $\Sigma$ is compact with $\mathrm{Ric}_{g_\Sigma}\geq 0$.
    \end{theorem}
We first need some preliminaries. We will carefully adapt the arguments of \cite{HongWang} to our setting.

\begin{lemma}\label{lemNewu}
Let $n\geq 2$. Let $(M^n,g)$ be a possibly noncompact manifold with boundary with unit outward normal $\nu$. Let $\gamma\geq 0$ and assume there is $u\in C^{2,\alpha}(M)$, $u>0$, such that 
\[
        -\gamma\Delta u+\mathrm{Ric} \cdot u\geq 0, \quad \text{on $M$}, \qquad H_{\partial M}+\gamma u^{-1}\langle \nabla u,\nu\rangle \geq 0, \quad \text{on $\partial M$},
\]
where $H_{\partial M}$ is computed with respect to the outward unit normal $\nu$ to $\partial M$.

Then for any $p_0 \in M$, there exist positive constants 
$r_0 = r_0(n,M,p_0)$ and $t_0 = t_0(n,M,p_0)$ such that the following holds. For any $r \leq r_0$, any $p \in \overline{B_{2r}(p_0)}$, and any $0 < t \leq t_0$ there exists a positive function $u_{r,t}\in C^{2,\alpha}(M)$ with the following conditions:
\begin{enumerate}[label=(\roman*)]
    \item $u_{r,t} \to u$ in the $C^2$ sense as $r,t \to 0^+$, and $u_{r,t} \to u$ in $C^{2,\alpha}$ for any fixed $r>0$ as $t \to 0^+$.
    \item $u_{r,t} = u$ on $M \setminus B_{3r}(p)$, and $u_{r,t} < u$ in $B_{3r}(p)$.
    \item $-\gamma \Delta u_{r,t} + \mathrm{Ric}\cdot u_{r,t} > 0 \quad \text{in } B_{3r}(p)\setminus B_r(p)$.
    \item $H+\gamma u_{r,t}^{-1}\langle \nabla u_{r,t},\nu\rangle \geq 0 \quad \text{on } \partial M$.
    \item $u_{r,t} > (1-t)  u > \frac{1}{10}u$ on $M$.
\end{enumerate}
\end{lemma}
\begin{proof}
    It follows from \cite[Lemma 3.3]{HongWang}. The inequality $-\gamma\Delta u+\mathrm{Ric}\cdot u\geq 0$ enters into play only at Line 8 from the bottom of page 11 in \cite{HongWang}. The more general version of item (iv) above comes from the computations in \cite{HongWang}. In particular, using the first computation at page 12 in \cite{HongWang}, calling $v_r$ the function constructed there, we have 
    \[
    H+\gamma u_{r,t}^{-1}\langle \nabla u_{r,t},\nu\rangle = H+\gamma u^{-1}\langle \nabla u,\nu\rangle + \gamma\frac{t}{1+tv_r}\langle \nabla v_r,\nu\rangle\geq 0.
    \]
    The last item (v) follows immediately from the construction  in \cite{HongWang}, if $r_0$ is chosen small enough.
\end{proof}

{
\begin{lemma}\label{lem:InteriorDoesNotTouchBoundaryAndOtherProperties}
Let $n\geq 2$. Let $(M^n,g)$ be a possibly noncompact manifold with disconnected boundary. Assume at least one boundary component is compact and denote it as $\Sigma_1$. Let $w\in C^{2,\alpha}(M)$ be a positive function. For a compact Lipschitz curve $\beta$ in $M$, denote  $E_w(\beta) := \int_\beta w \mathrm{d}\ell_g$, and define
\begin{equation}\label{eqn:MatchalCDefinition}
\mathcal{C}:=\Big\{\beta\in \mathrm{Lip}_g([a,b];M):\beta(a)\in \Sigma_1,\, \beta(b) \in \partial M \setminus \Sigma_1\Big\}.
\end{equation}
Let $\{\sigma_i\}\subset\mathcal{C}$ be such that $\lim_{i\to \infty} E_w(\sigma_i)=\inf_{\beta\in\mathcal{C}} E_w(\beta)=:m<\infty$. Parameterize $\sigma_i:[0,L_i]\to M$ by $g$-arclength. Assume $L=\lim L_i>0$ exists; let $I:=[0,L]$ if $L<\infty$, and $I:=[0,\infty)$ otherwise.
Suppose $\sigma_i$ converges to $\sigma:I\to M$ locally uniformly\footnote{Notice $\Sigma_1$ is compact. Hence, by Arzelà--Ascoli, given $\{\sigma_i\}$ we can always assume the existence of $L$ and $\sigma$, up to extracting subsequences. Moreover, since $d_g(\Sigma_1,\partial M\setminus \Sigma_1)>0$, we have $L>0$.}. We have the following.
\begin{enumerate}
    \item $\sigma(0)\in\Sigma_1$, and $E_w(\sigma)\leq E_w(\beta)$, for every $\beta\in\mathcal{C}$;
    \item $\sigma(\mathrm{int}(I))\cap \partial M=\emptyset$;
     \item For every $R<\sup I$, $\sigma|_{[0,R]}$ minimizes $E_w$ among compact curves connecting $\Sigma_1$ to $\sigma(R)$;
     \item If $L<\infty$, then $\sigma(L)\in \partial M\setminus\Sigma_1$. In particular, $\sigma\in\mathcal{C}$ is a minimizer of $E_w$ in $\mathcal{C}$;
     \item If $L=\infty$, then $\sigma$ has infinite length.
\end{enumerate}
As a consequence, $\sigma$ is a free boundary critical point of $E_w$, and thus it is $C^{3,\alpha}$-regular.
\end{lemma}
\begin{proof}
First, notice that by the assumptions, and since $\Sigma_1$ is closed, it immediately follows $\sigma(0)\in \Sigma_1$. 

Let us prove Item (1). We have that $\sigma:I\to M$ is $1$-Lipschitz with respect to $d_g$, because it is the limit of $1$-Lipschitz functions with respect to $d_g$. Then $\|\sigma'\|_g\leq 1$ almost everywhere. Let $T<\sup I$. For $i$ large enough, $T<L_i$. Thus, using Fatou's Lemma, we have that for every $\beta\in\mathcal{C}$,
\begin{equation}\label{eqn:Fatou'sLemma}
\begin{aligned}
E_w(\sigma|_{[0,T]})&=\int_0^T w(\sigma(t))\|\sigma'\|_g\mathrm{d}t \leq \int_0^T w(\sigma(t))\mathrm{d}t \leq \liminf_{i\to\infty} \int_0^T w(\sigma_i(t))\mathrm{d}t, \\
&\leq \liminf_{i\to\infty} \int_0^{L_i} w(\sigma_i(t))\mathrm{d}t \leq \liminf_{i\to\infty} E_w(\sigma_i) = m \leq E_w(\beta).
\end{aligned}
\end{equation}
Taking $T\to \sup I$ we get the sought conclusion.
\smallskip

Let us prove Item (2). Assume for the sake of contradiction that $p \in \sigma(\mathrm{int}(I))\cap \partial M$; therefore, there exists $0<T<\sup I$ such that $p=\sigma(T)\in\partial M$. There are two cases.

\textbf{Case 1}. $p\in\Sigma_1$. We aim at showing that this case is not possible.

For sufficiently large $i$, $T<L_i$. Let $\tilde{\sigma}_i$ be the concatenation of a $g$-minimizing geodesic from $p$ to $\sigma_{i}(T)$ (denoted as $\tau_i$), and $\sigma_{i}|_{[T,L_i]}$. We note that $\tilde{\sigma}_i\in \mathcal{C}$ and so $m\leq E_w (\tilde{\sigma}_i)$. We estimate
\begin{equation}\label{eqn:InequalityOnE}
m\leq E_w (\tilde{\sigma}_i)=E_w (\sigma_{i})-E_w (\sigma_{i}|_{[0,T]}) + E_w (\tau_i).
\end{equation}
We have that $\sigma_{i}|_{[0,T]} \subset \{p\in M:d_g(\Sigma_1,p)\leq T\}$. Indeed, for any $t\in [0,T]$
\[
d_{g}(\sigma_{i}(t),\Sigma_1)\leq d_g(\sigma_{i}(0), \Sigma_1)+d_g(\sigma_{i}(t), \sigma_{i}(0)) \leq t\leq T,
\]
where in the third inequality we use that $\sigma_{i}$ is 1-Lipschitz with respect to $d_g$. We have that $\tau_i \subset \{p\in M:d_g(\Sigma_1,p)\leq T+\mathrm{diam}(\Sigma_1)\}$. Indeed for any $i$ and any $t\in\mathrm{dom}(\tau_i)$ we have
\begin{align*}
    d_{g}(\tau_i(t),\Sigma_1)&\leq d_g(p, \Sigma_1)+d_g(\tau_i(t), p) \\
    &\leq d_g(p, \sigma_{i}(T))\\
    &\leq d_g(p, \sigma_{i}(0))+ d_g(\sigma_{i}(T), \sigma_{i}(0)) \\
    &\leq \mathrm{diam}_g(\Sigma_1)+T,
\end{align*}
where the second inequality follows since $\tau_i$ is a $g$-minimizing geodesic from $p$ to $\sigma_{i}(T)$ and the last inequality follows since $\sigma_{i}$ is 1-Lipschitz with respect to $d_g$.

Define the following
\begin{align*}
    &C^*:=\sup\{w(q):q\in \{p\in M:d_g(\Sigma_1,p)\leq T+\mathrm{diam}(\Sigma_1)\} \}\\
    &C_*:=\inf\{w(q):q\in \{p\in M:d_g(\Sigma_1,p)\leq T+\mathrm{diam}(\Sigma_1)\} \}.
\end{align*}
Since $\{p\in M:d_g(\Sigma_1,p)\leq T+\mathrm{diam}(\Sigma_1)\}$ is compact, we have that $C^*$ is finite, and $C_*$ is positive. Furthermore we have 
\begin{equation}\label{eqn:InequalityOnE2}
E_w(\sigma_{i}|_{[0,T]})\geq C_*\ell_g(\sigma_{i}|_{[0,T]}) = C_*T, \qquad E_w(\tau_i)\leq C^* d_g(\sigma(T),\sigma_{i}(T)).
\end{equation}
Therefore, by \eqref{eqn:InequalityOnE}, and \eqref{eqn:InequalityOnE2}, we have
\[
m\leq E_w(\sigma_{i}) - C_*T + C^*d_g(\sigma(T),\sigma_{i}(T)).
\]
We know that $\sigma_{i}$ converges to $\sigma$ uniformly on compact sets, and $E_w(\sigma_{i})\to m$. Taking $i\to\infty$ in the previous inequality leads to
\[
m\leq m-C_*T<m,
\]
a contradiction, as desired.

\textbf{Case 2}. $p\in \partial M\setminus \Sigma_1$. In this case, using Item (1), $E_w(\sigma|_{[0,T]})\leq E_w(\sigma)\leq E_w(\beta)$ for every $\beta\in\mathcal{C}$. Thus $\sigma|_{[0,T]}\in\mathcal{C}$ is a minimizer of $E_w$ in $\mathcal{C}$.

We claim that $I=[0,T]$. If not, $I\setminus [0,T]\neq \emptyset$. Since we have by Item (1) that $E_w(\sigma)\leq E_w(\sigma|_{[0,T]})$, we deduce $E_w(\sigma|_{I\setminus[0,T]})=0$. Hence we have that $\sigma|_{I\setminus [0,T]}\equiv p$. Fix $T^*>T$ with $[T,T^*]\subset I$. Let $\tau_{i,1}$ be a $g$-minimizing geodesic joining $\sigma_i(T)$ to $p$, and $\tau_{i,2}$ be a $g$-minimizing geodesic joining $p$ to $\sigma_i(T^*)$. Arguing as in \textbf{Case 1} above, there exists $\eta>0$ such that $E_w(\sigma_i|_{[T,T^*]})-E_w(\tau_{i,1})-E_w(\tau_{i,2})>\eta$ for $i$ large enough. Thus, for $i$ large enough, the concatenation of $\sigma_i|_{[0,T]},\tau_{i,1},\tau_{i,2},\sigma_i|_{[T^*,L_i]}$, which is in $\mathcal{C}$, will have energy $E_w$ strictly less than $m$, contradiction.

Hence, in this case, we have proved $I=[0,T]$, and thus $T=L$ and $L$ must be finite. Let us show that $\sigma((0,L))\subset M\setminus\partial M$. If by contradiction there is $0<T'<L$ such that $\sigma(T')\in\partial M$ we have two cases. If $\sigma(T')\in\Sigma_1$, we find a contradiction using \textbf{Case 1}. If $\sigma(T')\in \partial M\setminus\Sigma_1$, we find a contradiction as in the paragraph above. This concludes the proof of Item (2). 
\smallskip

Let us prove Item (3) by contradiction.
If Item (3) is not true, there is $R<\sup I$ and a compact curve $\tau$ connecting $\Sigma_1$ to $\sigma(R)$ such that $E_w(\tau)<E_{w}(\sigma|_{[0,R]})$. Again, for $i$ large enough, $R<L_i$. Let $\rho_i$ be a $g$-minimizing geodesic connecting $\sigma(R)$ to $\sigma_{i}(R)$. Let $\xi_i$ be the concatenation of $\tau$, $\rho_i$, and $\sigma_{i}|_{[R,L_i]}$. We note that $\xi_i\in \mathcal{C}$ and so, since $m=\inf_{\beta\in \mathcal{C}} E_w(\beta)$,
\[
E_w(\sigma_i)+o_i(1)=m \leq E_{w}(\xi_i) = E_{w}(\tau)+E_{w}(\rho_i)+E_{w}(\sigma_{i}|_{[R,L_i]}).
\]
Thus,
\[
E_{w}(\sigma_i) - E_{w}(\sigma_{i}|_{[R,L_i]}) = E_{w}(\sigma_{i}|_{[0,R]})\leq E_{w}(\tau)+E_{w}(\rho_i)+o_i(1).
\]
Let $C:=\sup\{w(q):d_g(q,\sigma(R))\leq 1\}$. Since $\sigma_{i}(R)\to \sigma(R)$, we have $\rho_i \subset \{q\in M: d_g(q,\sigma(R)) \leq 1\}$ (for $i$ large enough), which is a compact set.
Putting everything together, and using Fatou's lemma as in \eqref{eqn:Fatou'sLemma}, we have
\begin{align*}
     E_{w}(\sigma|_{[0,R]})&\leq \liminf_{i\to\infty}E_{w}(\sigma_{i} | _{[0,R]})\\
     &\leq E_{w}(\tau)+ \liminf_{i}E_{w}(\rho_i)\\
     & \leq E_{w}(\tau) + \liminf_{i} Cd_g(\sigma_{i}(R),\sigma(R))\\
     &=E_{w}(\tau)\\
     &< E_{w}(\sigma|_{[0,R]}),
\end{align*}
which is a contradiction.
\smallskip

Let us prove Item (4).  Since $\sigma_i(L_i)\in\partial M\setminus\Sigma_1$, and $\sigma_i(L_i)\to\sigma(L)$, we get that $\sigma(L)\in\partial M\setminus\Sigma_1$. Thus $\sigma\in\mathcal{C}$. Taking into account Item (1), we get that $\sigma\in\mathcal{C}$ is a minimizer of $E_w$ in $\mathcal{C}$.
\smallskip 

Let us finally prove Item (5). If by contradiction $\sigma$ has finite length, then there is a compact set $K$ such that $\sigma\subset K$. Let $\overline{B_1(K)}$ be the set of points with $g$-distance $\leq 1$ from $K$.  Let $C:=\inf_{\overline{B_1(K)}} w>0$. Fix $T>0$. For $i$ large enough, $\sigma_i|_{[0,T]}\subset \overline{B_1(K)}$ because $\sigma_i\to \sigma$ locally uniformly. Thus $E_w(\sigma_i)\geq E_w(\sigma_i|_{[0,T]})\geq CT$. Sending $i\to\infty$ we get $m\geq CT$. Since this is true for every $T>0$, we get $m=+\infty$, contradiction. Hence $\sigma$ has infinite length, as desired.
\end{proof}

\begin{lemma}\label{lem:existenceofminimizers}
Let $n\geq 2$. Let $(M^n,g)$ be a possibly noncompact manifold with disconnected boundary. Assume at least one boundary component is compact and denote it as $\Sigma_1$. Let $\nu$ be the outward unit normal to the boundary. 

Let 
\begin{equation}\label{eqn:AssumptionGammExistence}
\text{$0<\gamma\leq \frac{n-1}{n-2}$ if $n\geq 3$, or $0<\gamma<4$ if $n=2$}.
\end{equation}
Assume there is $u\in C^{2,\alpha}(M)$, $u>0$, such that 
        \begin{equation}\label{eqn:EitherOfTwo}
        -\gamma\Delta u+\mathrm{Ric}\cdot u \geq 0, \quad \text{on $M$}, \qquad H_{\partial M}+\gamma u^{-1}\langle \nabla u,\nu\rangle\geq 0, \quad \text{on $\partial M$},
        \end{equation}
where $H_{\partial M}$ is computed with respect to the outward unit normal $\nu$ to $\partial M$. When \(n\geq 3\), assume in addition that either
\(\gamma<\frac{n-1}{n-2}\), or at least one of the two inequalities in
\eqref{eqn:EitherOfTwo} is strict.

For a compact Lipschitz curve $\beta$ in $M$, let $E(\beta) := \int_\beta u^\gamma \mathrm{d}\ell_g$. Then there exists $\sigma\in \mathcal{C}$ (defined in \eqref{eqn:MatchalCDefinition}) such that $E(\sigma)\leq E(\sigma')$ for all $\sigma'\in \mathcal{C}$.
    
\end{lemma}
\begin{proof}
   Let $m=\inf_{\sigma\in \mathcal{C}} E(\sigma)$. Let $\{\sigma_i\}_{i=1}^\infty \subset \mathcal{C}$ be a sequence such that $E(\sigma_i) \to m$ as $i\to \infty$. Parameterize $\sigma_i$ by $g$-arclength; therefore, there are $L_i$ such that $\sigma_i:[0,L_i]\to M$ are $1$-Lipschitz with respect to the distance $d_g$. Let $L=\sup_i L_i>0$ and note $L$ is either finite or infinite. If $L<\infty$, define $I:=[0,L]$; otherwise $I:=[0,\infty)$. By Arzelà--Ascoli, up to subsequences we can assume $L=\lim_i L_i$, and there exists $\sigma:I\to M$ such that $\sigma_i\to\sigma$ locally uniformly. 

   We claim that $L$ cannot be infinite. If it were, by \cref{lem:InteriorDoesNotTouchBoundaryAndOtherProperties} we would have:
\begin{itemize}
    \item $\sigma(0)\in\Sigma_1$, and $\sigma((0,\infty))\cap \partial M = \emptyset$;
    \item $\sigma$ has infinite $g$-length, and $E(\sigma)<+\infty$;
    \item for every $R>0$, $\sigma|_{[0,R]}$ is $E$-minimizing among compact curves connecting $\Sigma_1$ to $\sigma(R)$.
\end{itemize}
The latter properties contradict \Cref{lemfiniteenergyfintelength}, and so $L$ cannot be infinite.

Hence $L$ is finite. Thus $\sigma\in\mathcal{C}$ is a minimizer of $E$ among curves in $\mathcal{C}$ by Item (4) of \cref{lem:InteriorDoesNotTouchBoundaryAndOtherProperties}, as desired.
\end{proof}
}

\begin{proposition}\label{propNearbyMinimizer}
     Let $n\geq 2$. Let $(M^n,g)$ be a possibly noncompact manifold with disconnected boundary. Assume at least one boundary component is compact and denote it as $\Sigma_1$. Let $\nu$ be the outward unit normal to the boundary. Let $0<\gamma<\frac{n-1}{n-2}$ if $n\geq 3$, or $0<\gamma<4$ if $n=2$. Assume there is $u\in C^{2,\alpha}(M)$, $u>0$, such that 
        \[
        -\gamma\Delta u+\mathrm{Ric}\cdot u \geq 0, \quad \text{on $M$}, \qquad H_{\partial M}+\gamma u^{-1}\langle \nabla u,\nu\rangle\geq 0, \quad \text{on $\partial M$},
        \]
        where $H_{\partial M}$ is computed with respect to the outward unit normal $\nu$ to $\partial M$. For a compact Lipschitz curve $\beta$ in $M$, let $E(\beta) := \int_\beta u^\gamma \mathrm{d}\ell_g$.

        Assume that  there exists a minimizer $\mathcal{C}\ni\sigma:[0,1]\to M$ of $E$ among competitors in the class $\mathcal{C}$ defined in \eqref{eqn:MatchalCDefinition}. Then, for any $s \in [0,1]$, there exists a positive constant $0<r_0 = r_0(s,\gamma)$ such that the following holds. For every $0 < r \leq r_0$ and every $p \in M \setminus \sigma$ with
\[
d_g(p,\sigma) = d_g(p,\sigma(s)) = 2r,
\]
there exists $\sigma_r\in \mathcal{C}$ such that:
\begin{enumerate}
    \item $\sigma_r$ is a minimizer of $E$ among competitors in the set $\mathcal{C}$;
    \item $\sigma_r\cap \overline{B_r(p)}\neq\emptyset$.
\end{enumerate}
\end{proposition}
\begin{proof}
    Fix $s\in [0,1]$ and let $p_0=\sigma(s)$. Let $r_0,t_0>0$ be the constants from \Cref{lemNewu}. Take $0<r<r_0$ and $p$ to be such that
\[
d_g(p,\sigma) = d_g(p,p_0) = 2r.
\]
For any $0<t<t_0$, let $u_{r,t}$ be the function constructed in \Cref{lemNewu}. For any $h>10r_0$, let $u_{r,t,h}$ be a $C^{2,\alpha}$ function that satisfies
\begin{align*}
    &u_{r,t,h} = u_{r,t} \text{ in }  \{q\in M : d_g(q,\sigma)\leq 2h\},\\
    &u_{r,t,h} \geq u_{r,t} \text{ in }  M\setminus\{q\in M : d_g(q,\sigma)\leq 2h\},\\
    & u_{r,t,h} \geq \max\{1,u_{r,t}\} \text{ in }  \{q\in M : d_g(q,\sigma)\geq 2h+1\}.
\end{align*}
Let $\mathcal{C}$ be defined as in \eqref{eqn:MatchalCDefinition}.
 For every $\beta\in\mathcal{C}$, $0<r<r_0$, $0<t<t_0$, and $h>10r_0$, let 
 $$
 E_{r,t,h}(\beta):=\int_\beta u^\gamma_{r,t,h}\mathrm{d}\ell_g,
 $$ and $m_{r,t,h}:=\inf_\mathcal{\beta\in C} E_{r,t,h}(\beta)<\infty$. Let $\{\sigma^i_{r,t,h}\}_{i=1}^\infty \subset \mathcal{C}$ be a sequence such that $E_{r,t,h}(\sigma^i_{r,t,h}) \to m_{r,t,h}$ as $i\to \infty$. Parameterize $\sigma^i_{r,t,h}$ by $g$-arclength; therefore, there are $L_i$ such that $\sigma^i_{r,t,h}:[0,L_i]\to M$ are $1$-Lipschitz with respect to $d_g$, with $\sigma_{r,t,h}^i(0)\in\Sigma_1$. Let $L=\sup_i L_i>0$ and note $L$ is either finite or infinite. If $L<\infty$, let $I:=[0,L]$; otherwise $I:=[0,\infty)$. Up to subsequence, using Arzelà--Ascoli, we can assume $L=\lim_i L_i$ and $\sigma_{r,t,h}^i\to \sigma_{r,t,h}$ locally uniformly, where $\sigma_{r,t,h}:I\to M$.

We claim that $L$ cannot be infinite. Indeed, by construction there is a positive constant $C>0$ such that $u_{r,t,h}\geq C>0$ on $M$. Thus, if $L$ is infinite, $E_{r,t,h}(\sigma_{r,t,h}^i)\geq C^\gamma L_i\to\infty$, which is a contradiction. Thus, $L$ is finite and by Item (4) of \cref{lem:InteriorDoesNotTouchBoundaryAndOtherProperties}$, {\sigma}_{r,t,h}\in\mathcal{C}$ is a minimizer of $E_{r,t,h}$ among competitors in $\mathcal{C}$. We claim the following.

\begin{claim}
   For any $0<r<r_0$, $0<t<t_0$,  there exists a constant $A$ such that  we have $\int_{\sigma_{r,t,h}\cap B_{3r}(p)}\, \mathrm{d}\ell_g\leq A$ for every $h>10r_0$.
\end{claim}
\begin{proof}[Proof of Claim]
    First, we note that
    \begin{equation}\label{eqn:InequalityErth}
    E_{r,t,h}(\sigma_{r,t,h})\leq E_{r,t,h}(\sigma) = E_{r,t}(\sigma) \leq E(\sigma),
    \end{equation}
    where the first inequality follows by the minimizing property of $\sigma_{r,t,h}$, the second equality follows because $u_{r,t,h}=u_{r,t}$ on $\sigma$, and the last inequality follows from item (ii) of \Cref{lemNewu}. Now note,
    \begin{align*}
         \inf_{B_{3r}(p)}\left(u_{r,t}^\gamma\right) \cdot \int_{\sigma_{r,t,h}\cap B_{3r}(p)}\, \mathrm{d}\ell_g &\leq   \int_{\sigma_{r,t,h}\cap B_{3r}(p)} u_{r,t}^\gamma\, \mathrm{d}\ell_g \\
         &\leq \int_{\sigma_{r,t,h}\cap B_{3r}(p)} u_{r,t,h}^\gamma\, \mathrm{d}\ell_g \\
         & \leq E_{r,t,h}(\sigma_{r,t,h})\\
         &\leq E(\sigma).
    \end{align*}
    Therefore, it is enough to take $A:=\frac{E(\sigma)}{ \inf_{B_{3r}(p)}\left(u_{r,t}^\gamma\right)}$.
\end{proof}

\begin{claim}
    For every $0<r<r_0$, $0<t<t_0$, there exists $\varepsilon>0$ small enough such that $\sigma_{r,t,h}\cap B_{3r-\eps}(p) \neq \emptyset$ for every $h\geq 10r_0$.
\end{claim}
\begin{proof}[Proof of Claim]
    By way of contradiction, suppose not. By the construction of $u_{r,t}$, see \cite[Lemma 3.3]{HongWang}, we have $u_{r,t}^\gamma\ge (1-C\eps)u^\gamma $ in $B_{3r}(p)\backslash B_{3r-\eps}(p)$ for some positive constant $C$, and for all $\varepsilon$ small enough.

If $\sigma_{r,t,h}\cap B_{3r-\eps}(p)=\emptyset$ for some $h\geq 10r_0$, then 
\begin{align*}
0<\;&\int_\sigma u^\gamma \, \mathrm{d}\ell_g-\int_\sigma u_{r,t}^\gamma \, \mathrm{d}\ell_g \\
=\;& \int_\sigma u^\gamma \, \mathrm{d}\ell_g-\int_{\sigma} u_{r,t,h}^\gamma \, \mathrm{d}\ell_g\\
\leq\;& \int_\sigma u^\gamma \, \mathrm{d}\ell_g-\int_{\sigma_{r,t,h}} u_{r,t,h}^\gamma\mathrm{d}\ell_g \\
\leq\;& \int_\sigma u^\gamma \, \mathrm{d}\ell_g
- \int_{\sigma_{r,t,h}\cap B_{3r}(p)} (1-C\eps)u^\gamma \, \mathrm{d}\ell_g- \int_{\sigma_{r,t,h}\setminus B_{3r}(p)} u^\gamma \, \mathrm{d}\ell_g\\
=\;& \int_\sigma u^\gamma\, \mathrm{d}\ell_g
-\int_{\sigma_{r,t,h}} u^\gamma \, \mathrm{d}\ell_g
\quad + C\eps \int_{\sigma_{r,t,h}\cap B_{3r}(p)} u^\gamma \, \mathrm{d}\ell_g\\
\leq\;& \int_\sigma u^\gamma \, \mathrm{d}\ell_g
-\int_\sigma u^\gamma \, \mathrm{d}\ell_g
+ C\eps \int_{\sigma_{r,t,h}\cap B_{3r}(p)} u^\gamma \, \mathrm{d}\ell_g\\
=\;& C\eps \int_{\sigma_{r,t,h}\cap B_{3r}(p)} u^\gamma \, \mathrm{d}\ell_g\\
\leq \; & AC\sup_{B_{3r}(p)} (u^\gamma) \cdot \eps.
\end{align*}
The first line follows from the fact $\sigma\cap B_{3r}(p)$ has positive $1$-measure, and item (ii) of \Cref{lemNewu}. The equality on the second line holds since $u_{r,t,h}=u_{r,t}$ on $\sigma$. The third line follows from the minimality of $\sigma_{r,t,h}$. The fourth line follows since $B_{3r}(p) \subset \{q\in M:d_g (q,\sigma)\leq 2h\}$, $u_{r,t,h}\geq u_{r,t}=u$ on $M\setminus B_{3r}(p)$; and, by the contradiction assumption, $\sigma_{r,t,h}\cap B_{3r-\eps}(p)=\emptyset$ so that $u^\gamma_{r,t}\geq (1-C\eps)u^\gamma$ on $\sigma_{r,t,h}\cap B_{3r}(p)$. Finally, the sixth line follows from minimality of $\sigma$, {and since $\sigma_{r,t,h}\in\mathcal{C}$}. Therefore, if we take $\eps$ small enough (independent on $h$) we arrive at a contradiction.
\end{proof}

{Let us take $h_i\to\infty$. Let us parameterize $\sigma_{r,t,h_i}:[0,L_{h_i}]\to M$ by $g$-arclength. Up to relabeled subsequences, and using Arzelà--Ascoli, $L=\lim_i L_{h_i}>0$ exists, and $\sigma_{r,t,h_i}$ converges locally uniformly to a $1$-Lipschitz (with respect to $d_g$) curve $\sigma_{r,t}:J\to M$. Here $J:=[0,L]$ when $L<\infty$, and $J:=[0,\infty)$ otherwise.

\begin{claim}\label{claim:minimizingErt}
  Let $m:=\inf_{\beta\in\mathcal{C}}E_{r,t}(\beta)$. Then we have $\lim_i E_{r,t}(\sigma_{r,t,{h_i}}) = m$.
\end{claim} 
\begin{proof}[Proof of Claim]
   For any $\eps>0$, there exist $\beta_\eps\in\mathcal{C}$ such that $m\leq E_{r,t}(\beta_\eps)\leq m+\eps$. Now, choose $h_i$ (depending on $\eps$) large enough that $\beta_\eps\subset \{q\in M:d_g(q,\sigma)\leq 2h_i\}$. Recall that $\sigma_{r,t,h_i}\in \mathcal{C},\beta_\eps\in \mathcal{C}$, and $u_{r,t}\leq u_{r,t,h_i}$ on $M$. Therefore,
   \[
   m\leq E_{r,t}(\sigma_{r,t,h_i})\leq E_{r,t,h_i}(\sigma_{r,t,h_i})\leq E_{r,t,h_i}(\beta_\eps)=E_{r,t}(\beta_\eps) \leq m+\eps,
   \]
   and this proves the claim, as desired.
\end{proof}

Therefore, $\sigma_{r,t,h_i}$ and $\sigma_{r,t}$ satisfy the hypotheses of \Cref{lem:InteriorDoesNotTouchBoundaryAndOtherProperties}}. Then, by \Cref{lem:InteriorDoesNotTouchBoundaryAndOtherProperties}, we have  that 
\begin{itemize}
    \item $\sigma_{r,t}$ cannot intersect $\partial M$ at interior points of its support;
    \item If $J$ is compact, $\sigma_{r,t}$ must connect $\Sigma_1$ to $\partial M\setminus\Sigma_1$ and $\sigma_{r,t}\in \mathcal{C}$ is a minimizer of $E_{r,t}$ among competitors in $\mathcal{C}$.
\end{itemize}  

Let us now prove the following claim.
\begin{claim}
    $\sigma_{r,t} \cap \overline{B_r(p)} \neq \emptyset$.
\end{claim}
\begin{proof}[Proof of Claim]
Suppose for the sake of contradiction $\sigma_{r,t}\cap\overline{B_r(p)}=\emptyset$. Because of the contradiction assumption and item (ii) and (iii) of \cref{lemNewu}, we have that $-\gamma \Delta u_{r,t} + \Ric \cdot u_{r,t} \geq 0$ on the support of $\sigma_{r,t}$. Moreover, by item (iv) of \cref{lemNewu} we have $H+\gamma u_{r,t}^{-1}\langle \nabla u_{r,t},\nu\rangle\geq 0$ on $\partial M$. We must deal with two cases: 

\textbf{Case 1}. We have $\sigma_{r,t}:[0,L]\to M$, with $L<\infty$.

In this case, as noted above, we have that $\sigma_{r,t}\in\mathcal{C}$ minimizes $E_{r,t}$ over $\mathcal{C}$, and $\sigma_{r,t}((0,L))\subset M\setminus\partial M$. Moreover, recall that $\sigma_{r,t,h_i}:[0,L_{h_i}]\to M$, $L=\lim_i L_{h_i}$ is finite, and $\sigma_{r,t,h_i}\to \sigma_{r,t}$ uniformly. Thus, taking into account that $\sigma_{r,t,h_i}\cap B_{3r-\varepsilon}(p)\neq\emptyset$ for all $i$ large enough, and for all $\varepsilon$ small enough, we also deduce $\sigma_{r,t}\cap B_{3r-\varepsilon}(p)\neq \emptyset$ for all $\varepsilon>0$ small enough. 

Now, using \cref{lemgradienvanishes}, we conclude that $-\gamma \Delta u_{r,t} + \Ric \cdot u_{r,t} = 0$ along $\sigma_{r,t}$, a contradiction with item (iii) of \cref{lemNewu}, taking into account that $\sigma_{r,t}\cap B_{3r-\varepsilon}(p)\neq \emptyset$, for all $\varepsilon>0$ small enough.

\textbf{Case 2}. We have $\sigma_{r,t}:[0,\infty)\to M$.

By  \Cref{lem:InteriorDoesNotTouchBoundaryAndOtherProperties}, we also have that, for every $R>0$, $\sigma_{r,t}|_{[0,R]}$ is $E_{r,t}$-minimizing among compact curves connecting $\Sigma_1$ to $\sigma_{r,t}(R)$. Moreover, $E_{r,t}(\sigma_{r,t})<\infty$, and $\sigma_{r,t}$ has infinite length. Furthermore, as noted above, $\sigma_{r,t}$ cannot intersect $\partial M$ at interior points of its support. The latter properties are in contradiction with \cref{lemfiniteenergyfintelength}, as desired.
\end{proof}

{
Let $t_i<t_0$ with $t_i\to 0$. We have constructed $\sigma_{r,t_i}:J_i\to M$, where $J_i$ is an interval with $0=\min J_i$, and $\sigma_{r,t_i}(0)\in\Sigma_1$. Up to subsequences we can assume $L_\infty:=\lim_i(\sup J_i)>0$ exists. We set $I_\infty:=[0,L_\infty]$ if $L_\infty<\infty$, and $I_\infty:=[0,\infty)$ otherwise.

By Arzelà--Ascoli, 
there is a (relabeled) subsequence $\sigma_{r,t_i}$ converging uniformly on compact subsets to a Lipschitz curve $\sigma_r:I_\infty\to M$. By construction, for every $t_i$ there is a sequence $h_{i,j}\to_j \infty$ such that $\sigma_{r,t_i,h_{i,j}}\to_j \sigma_{r,t_i}$ locally uniformly. We can extract a diagonal subsequence $\sigma_{r,t_i,h_{j_i}}:[0,L_{h_{j_i}}]\to M$ with $L_\infty=\lim_i L_{h_{j_i}}$ such that $\sigma_{r,t_i,h_{j_i}}\to\sigma_r$ locally uniformly. We claim the following
{{
\begin{claim}
    $\lim_i E(\sigma_{r,t_i,h_{j_i}}) = E(\sigma)=\min_{\beta\in\mathcal{C}} E(\beta)$.
\end{claim}
\begin{proof}[Proof of Claim]
    By \Cref{lemNewu} (ii) and (v), we have that $u_{r,t_i}\leq u\leq \frac{1}{1-t_i}u_{r,t_i}$. By the minimizing property of $\sigma$, the fact that $u_{r,t_i,h_{j_i}}\geq u_{r,t_i}$ on $M$, the minimizing property of $\sigma_{r,t_i,h_{j_i}}$, and the fact that $u_{r,t_i,h_{j_i}}=u_{r,t_i}$ on $\sigma$,  we have
    \[
    E(\sigma)\le E(\sigma_{r,t_i,h_{j_i}})\leq \frac{1}{(1-t_i)^\gamma} E_{r,t_i,h_{j_i}}(\sigma_{r,t_i,h_{j_i}})\leq \frac{1}{(1-t_i)^\gamma}E_{r,t_i,h_{j_i}}(\sigma) = \frac{1}{(1-t_i)^\gamma}E_{r,t_i}(\sigma).
    \]
    Now, by taking the limit as $t_i\to 0$, we have 
    \[
    E(\sigma)\leq \lim_{i\to\infty} E(\sigma_{r,t_i,h_{j_i}}) \leq \lim_{i\to\infty}\frac{1}{(1-t_i)^\gamma}E_{r,t_i}(\sigma)= E(\sigma),
    \]
as desired.
\end{proof}

Now, we aim at showing that $L_\infty<\infty$. By contradiction, $L_\infty=\infty$. Thus, by taking into account the previous claim, we can apply \Cref{lem:InteriorDoesNotTouchBoundaryAndOtherProperties} and conclude that:
\begin{itemize}
    \item $\sigma_r(0)\in \Sigma_1$ and $\sigma_r((0,\infty))\subset M\setminus\partial M$;
    \item For every $\beta\in\mathcal{C}$, $E(\sigma_r)\leq E(\beta)$. Thus, $E(\sigma_r)<\infty$;
    \item $\sigma_{r}|_{[0,R]}$ is $E$-minimizing among compact curves connecting $\Sigma_1$ to $\sigma_r(R)$;
    \item $\sigma_r$ has infinite $g$-length.
\end{itemize}
Hence, from the latter findings, we get a contradiction with \cref{lemfiniteenergyfintelength}. Thus $L_\infty<\infty$ as desired.}}

Since $L_\infty=\lim_i(\sup J_i)<\infty$, we have $\sigma_{r,t_i}:[0,\ell_i]\to M$, where $\ell_i$ are uniformly bounded. Moreover, $\sigma_{r,t_i}\to \sigma_{r}$ uniformly.
Hence, since by our previous claims we have $\sigma_{r,t}\cap\overline{B_r(p)}\neq\emptyset$ for all $t$ small enough, we conclude that $\sigma_r\cap \overline{B_r(p)}\neq \emptyset$. Finally, by applying Item (4) of \Cref{lem:InteriorDoesNotTouchBoundaryAndOtherProperties} (to the sequence $\{\sigma_{r,t_i,h_{j_i}}\}_{i=1}^\infty$), we have that $\sigma_r\in\mathcal{C}$ is a minimizer of $E$ among competitors in $\mathcal{C}$, as desired.
}
\end{proof}

\begin{proof}[Proof of \Cref{ThmRigid}]
Define $\Sigma_1$ to be a compact connected component of the boundary. Let $\mathcal{C}$ be defined as in \eqref{eqn:MatchalCDefinition}.
    Consider
\[
\Omega := \left\{ p \in M \;\middle|\;
\begin{aligned}
&\text{there exists $\sigma\in\mathcal{C}$}\,\text{ such that $p\in\sigma$} \\
&\hspace{0pt}\text{and $\sigma$ minimizes $E$ among competitors in $\mathcal{C}$}
\end{aligned}
\right\}.
\]
By \Cref{lem:existenceofminimizers}, $\Omega$ is nonempty.
We claim that $\Omega$ is closed. 

To see this consider a sequence $\{q_i\}_{i=1}^\infty\subset \Omega$ converging to $q$. Let $\mathcal{C}\ni\sigma_{i}:[0,L_i]\to M$ be an $E$-minimizer among competitors in $\mathcal{C}$, parameterized by $g$-arclength, such that $q_i\in\sigma_i$. We will show that there is $\sigma\in\mathcal{C}$ that is an $E$-minimizer among competitors in $\mathcal{C}$, such that $q\in\sigma$.

{
Up to subsequences, we can assume $L=\lim_i L_i>0$ exists. Denote $I:=[0,L]$ if $L<\infty$, and $I:=[0,\infty)$ otherwise. Up to relabeled subsequences, $\sigma_i$ converges locally uniformly to $\sigma_\infty:I\to M$. We claim that $L<\infty$. Indeed, if not, by \cref{lem:InteriorDoesNotTouchBoundaryAndOtherProperties}, we would have that:
\begin{itemize}
    \item for every $R>0$, $\sigma_\infty|_{[0,R]}$ is an $E$-minimizer among compact curves connecting $\Sigma_1$ to $\sigma_\infty(R)$, and $E(\sigma_\infty)<\infty$;
    \item $\sigma_\infty(\mathrm{int}(I))\cap\partial M = \emptyset$;
    \item $\sigma_\infty$ must have infinite $g$-length.
\end{itemize} 
The latter properties would be in contradiction, due to \cref{lemfiniteenergyfintelength}. Hence $I=[0,L]$ with $L<\infty$. 

Since $L<\infty$, we have that $L_i$ is uniformly bounded. Moreover $\sigma_i:[0,L_i]\to M$ converges uniformly to $\sigma:[0,L]\to M$, $q_i\in \sigma_i$, and $q_i\to q$. The latter conditions readily imply $q\in\sigma_\infty$. Moreover, by Item (4) of \cref{lem:InteriorDoesNotTouchBoundaryAndOtherProperties}, $\sigma_\infty\in\mathcal{C}$ is an $E$-minimizer among curves in $\mathcal{C}$. Since $q\in\sigma_\infty$, we have that $\Omega$ is closed, as desired.
 }

We now show that $\Omega=M$ by contradiction. Assume $M\setminus\Omega\neq \emptyset$. Let $p_1\in M\setminus \Omega$. Since $\Omega$ is closed there exists a nearest point $p_0\in \Omega$ to $p_1$. Let $\sigma$ be an $E$-minimizer connecting $\Sigma_1$ to $\partial M \setminus \Sigma_1$ through $p_0$, and $\alpha$ be a $g$-minimizing geodesic segment connecting $p_0$ to $p_1$, with $\alpha(0)=p_0$. As $p_0\in \Omega$ is the nearest point to $p_1$, we have, for sufficiently small $r>0$, that $B_{2r}(\alpha(2r))\cap\Omega=\emptyset$. Therefore,
\[
d_g(\alpha(2r),\sigma)=d_g(\alpha(2r), p_0)=2r.
\]

Now, apply \Cref{propNearbyMinimizer} with $r<r_0$ and conclude $\overline{B_r(\alpha(2r))}\cap \Omega\neq \emptyset$, which is a contradiction. Thus, $\Omega =M$.

In conclusion, for any $p\in M$, we can apply \Cref{lemgradienvanishes} on an $E$-minimizer connecting $\Sigma_1$ to $\partial M\setminus \Sigma_1$ through $p$. We thus conclude $|\nabla u|(p)=0$ for every $p\in M$. Therefore, $u$ is constant, and so $(M,g)$ has $\mathrm{Ric}\geq 0$ and $H\geq 0$. The rigidity follows from Kasue's classical result \cite[Theorem B(1)]{Kasue83}.
\end{proof}

\section{Proof of \cref{thm:Pi10} and \cref{thm:RigidityOrPi1Zero}}\label{sec:proofs}

In this section we complete the proof of the main theorems. The following lemma is essentially contained in \cite{Galloway}.
\begin{lemma}\label{lem:Galloway}
    Let $n\geq 2$. Let $(M^n,g)$ be a complete (possibly noncompact) connected Riemannian manifold with compact connected boundary $\partial M$. Let $\iota:\partial M\to M$ be the inclusion map. Then:
    \begin{enumerate}
        \item Either $\iota_*:\pi_1(\partial M)\to \pi_1(M)$ is onto;
        \item Or there is a Riemannian covering map $\pi:(\tilde M,\tilde g)\to (M,g)$ such that $\partial\tilde M$ has at least two connected components, and at least one of them is mapped isometrically by $\pi$ onto $\partial M$ (and thus it is compact).
    \end{enumerate}
\end{lemma}
\begin{proof}
    The proof is in the first lines of \cite[Proof of Lemma 2]{Galloway}.
\end{proof}
Let us start with the proof of \cref{thm:Pi10}.

\begin{proof}[Proof of \cref{thm:Pi10}]
    If $\gamma=0$, the result follows from \cite[Proposition 2.8]{FraserLi}, see also \cite{LawsonUnknot}. So, we can assume $\gamma>0$. In order to prove $\pi_1(M,\partial M)=0$, we have to show that $\partial M$ is connected, and $\iota_*:\pi_1(\partial M)\to \pi_1(M)$ is surjective. 
    
    Let us first show that $\partial M$ is connected. If this is not the case, let $\Sigma_1$ be a connected compact component of $\partial M$, which exists by assumption, and denote
\[
\mathcal{C}:=\Big\{\beta\in \mathrm{Lip}_g([a,b];M):\beta(a)\in \Sigma_1,\, \beta(b) \in \partial M \setminus \Sigma_1\Big\}.
\]
Let $E(\tau):=\int_\tau u^\gamma\mathrm{d}\ell_g$, for every $\tau\in \mathcal{C}$. By \cref{lem:existenceofminimizers}, 
there exists $\sigma\in\mathcal{C}$ such that $E(\sigma)\leq E(\tau)$ for every $\tau\in \mathcal{C}$. Notice that by the minimality property, $\sigma$ intersects $\partial M$ only at its endpoints. Thus, by plugging $\psi\equiv 1$ in the stability inequality for $\sigma$, given by \cref{lem:SecondVariationGeneral}, we get a contradiction because at least one among $-\gamma\Delta u+\mathrm{Ric}u\geq 0$ and $H_{\partial M}+\gamma u^{-1}\langle\nabla u,\nu\rangle\geq 0$ is a strict inequality by assumption. Thus, $\partial M$ is connected, as desired. Moreover, by assumption $\partial M$ is compact.

    Let us now show that $\iota_*:\pi_1(\partial M)\to \pi_1(M)$ is surjective. If not, by \cref{lem:Galloway}, there is a Riemannian covering map $\pi:(\tilde M,\tilde g)\to (M,g)$ such that $\partial\tilde M$ has at least two boundary connected components, of which at least one is compact. Call $\tilde u:=u\circ \pi$. We have $-\gamma\tilde{\Delta}\tilde u+\tilde u\tilde{\mathrm{Ric}}\geq 0$ on $\tilde{M}\setminus\partial \tilde{M}$; and $\tilde{H}_{\partial\tilde M}+\gamma\tilde u^{-1}\langle\tilde\nabla\tilde u,\tilde\nu\rangle\geq 0$ on $\partial\tilde M$, where $\tilde{\square}$ denotes the object $\square$ in $(\tilde M,\tilde g)$. Moreover, at least one of the latter two inequalities is strict by assumption.  
    
    Let $\tilde\Sigma_1$ be a compact connected component of $\partial\tilde{M}$. Let 
    \[
\tilde{\mathcal{C}}:=\Big\{\beta\in \mathrm{Lip}_{\tilde g}([a,b];\tilde M):\beta(a)\in \tilde\Sigma_1,\, \beta(b) \in \partial \tilde M \setminus \tilde\Sigma_1\Big\},
\]
and $\tilde E(\tau):=\int_\tau \tilde u^\gamma \mathrm{d}\ell_{\tilde g}$, for every $\tilde g$-Lipschitz curve $\tau$ in $\tilde M$. By \cref{lem:existenceofminimizers}, 
there exists $\tilde\sigma\in\tilde{\mathcal{C}}$ such that $\tilde E(\tilde\sigma)\leq \tilde E(\tau)$ for every $\tau\in \tilde{\mathcal{C}}$. Thus, we get a contradiction as above using \cref{lem:SecondVariationGeneral}.
Thus, we get that $\iota_*:\pi_1(\partial M)\to \pi_1(M)$ is surjective, as desired.
\end{proof}

We finally prove \cref{thm:RigidityOrPi1Zero}.
\begin{proof}[Proof of \cref{thm:RigidityOrPi1Zero}]
    Let us first deal with the case $n\geq 3$. If at least one of the inequalities in \eqref{eqn:TwoAssumptionsAgain} is a strict inequality, the assertion follows from \cref{thm:Pi10}.

    We are left to consider the case when $\gamma<\frac{n-1}{n-2}$, and both the inequalities in \eqref{eqn:TwoAssumptionsAgain} hold. In this case we want to show that either $\pi_1(M,\partial M)=0$, or there exists a compact Riemannian manifold $(\Sigma^{n-1},g_\Sigma)$, with $\mathrm{Ric}_{g_\Sigma}\geq 0$, and a Riemannian cover $\pi:([0,L]\times\Sigma,\mathrm{d}t^2+g_{\Sigma})\to (M^n,g)$ with $\mathrm{deg}(\pi)\leq 2$. 

    If $\pi_1(M,\partial M)\neq 0$ then we are left to consider two cases.

    \textbf{Case 1.} $\partial M$ is not connected. In this case \cref{thmSplitting} implies that there exists a compact Riemannian manifold $(\Sigma^{n-1},g_\Sigma)$, with $\mathrm{Ric}_{g_\Sigma}\geq 0$, such that $(M,g)\cong ([0,L]\times\Sigma^{n-1},\mathrm{d}t^2+g_\Sigma)$, as desired. In this case our assertion holds, and the Riemannian cover is actually a Riemannian isometry, so that $\deg(\pi)=1$.

    \textbf{Case 2.} $\partial M$ is connected, and $\iota_*:\pi_1(\partial M)\to\pi_1(M)$ is not surjective. By \cref{lem:Galloway} there is a Riemannian covering map $\pi:(\tilde M,\tilde g)\to (M,g)$ such that $\partial\tilde M$ has at least two boundary connected components, of which at least one is compact. By applying \cref{thmSplitting} with $\tilde u:=u\circ\pi$, we get that there exists a compact Riemannian manifold $(\Sigma^{n-1},g_\Sigma)$, with $\mathrm{Ric}_{g_\Sigma}\geq 0$, and such that $(\tilde M,\tilde g)\cong ([0,L]\times\Sigma^{n-1},\mathrm{d}t^2+g_{\Sigma})$. Moreover, in the covering construction of \cref{lem:Galloway} one boundary component maps
isometrically onto $\partial M$. Therefore the covering has degree at most $2$.

    The case $n=2$ follows arguing as in the two cases above, using \cref{thmSplitting}.
\end{proof}

\section{Appendix: Inradius and Bonnet–-Myers Type Estimates}\label{sec:app}
\begin{proposition}\label{prop:WeightedInradius}
Let \(n\geq 3\), let \(0\leq \gamma\leq \frac{n-1}{n-2}\), $k>0$, and let
\((M^n,g)\) be a smooth complete Riemannian manifold with nonempty boundary.
Let \(0<u\in C^{2,\alpha}(M)\), and assume that
\[
0<\inf_M u\leq \sup_M u<\infty .
\]
Assume that
\[
-\gamma\Delta u+\mathrm{Ric}\cdot u\geq 0
\qquad\text{on }M,
\]
and assume that, along \(\partial M\),
\[
H_{\partial M}+\gamma u^{-1}u_\nu\geq (n-1)k>0,
\]
where \(\nu\) is the outward unit normal to \(\partial M\),
\(u_\nu=\langle \nabla u,\nu\rangle\), and \(H_{\partial M}\) is computed
with respect to \(\nu\). Then
\[
\operatorname{InRad}(M):=\sup_{x\in M}d_g(x,\partial M)
\leq
\frac1k
\left(
\frac{\sup_M u}{\inf_M u}
\right)^{\frac{n-3}{n-1}\gamma}.
\]
Thus, if \(\partial M\) is compact, then \(M\) is compact.
\end{proposition}

\begin{proof}
Set
\[
\alpha:=\frac{n-3}{n-1}\gamma.
\]
Since $n\geq 3$ and $\gamma\geq 0$, we have $\alpha\geq 0$.

Fix $x\in \operatorname{int}M$. Consider the functional
\[
E(\beta):=\int_\beta u^\gamma\,\mathrm{d}\ell_g,
\]
where $\beta$ is a compact curve in $M$.
Equivalently, $E$ is the length functional of the conformally changed metric $
\widetilde g:=u^{2\gamma}g$.
Since \(u\) is bounded above and bounded below away from zero,
$(M,\widetilde g)$ is complete. Hence there exists an $E$-minimizing curve
\[
\sigma:[0,\ell]\to M,
\]
from $\partial M$ to $x$. We parameterize $\sigma$ by $g$-arclength, so that
\[
\sigma(0)\in \partial M,\qquad \sigma(\ell)=x,
\qquad |\sigma'|_g=1.
\]
In particular,
\[
d_g(x,\partial M)\leq \ell.
\]

Because $\sigma$ minimizes $E$, the weighted second variation inequality
applies with the initial endpoint free on $\partial M$ and the final endpoint
fixed at $x$. The same computations leading to Lemma~\ref{lem:SecondVariationGeneral} show that for every $\psi\in C^\infty([0,\ell])$, with $\psi(\ell)=0$, we have 
\[
\begin{aligned}
0\leq
&\int_0^\ell u^\alpha
\Bigg[
(n-1)(\psi')^2
+\psi^2
\Bigg(
\left(\frac{n-2}{n-1}\gamma^2-\gamma\right)u^{-2}(u')^2
-\gamma u^{-2}|\nabla^\perp u|^2  \\
&\qquad\qquad\qquad\qquad
+\gamma u^{-1}\Delta u-\Ric(\sigma',\sigma')
\Bigg)
\Bigg]\,\mathrm{d}t \\
&\qquad
-u(\sigma(0))^\alpha\psi(0)^2
\left(
H_{\partial M}+\gamma u^{-1}u_\nu
\right)(\sigma(0)).
\end{aligned}
\]
Now choose
\[
\psi(t):=1-\frac{t}{\ell}.
\]
Then
\[
\psi(0)=1,\qquad \psi(\ell)=0,\qquad \psi'=-\frac1\ell.
\]
Substituting this choice of $\psi$ gives
\[
\begin{aligned}
0&\leq
\frac{n-1}{\ell^2}\int_0^\ell u^\alpha\,\mathrm{d}t  \\
&+\int_0^\ell u^\alpha
\left(1-\frac{t}{\ell}\right)^2
\Bigg[
\left(\frac{n-2}{n-1}\gamma^2-\gamma\right)u^{-2}(u')^2
-\gamma u^{-2}|\nabla^\perp u|^2  \\
&\qquad\qquad\qquad\qquad
+\gamma u^{-1}\Delta u-\Ric(\sigma',\sigma')
\Bigg]\,\mathrm{d}t \\
&\qquad
-u(\sigma(0))^\alpha
\left(
H_{\partial M}+\gamma u^{-1}u_\nu
\right)(\sigma(0)).
\end{aligned}
\]
Hence, using the inequalities in the assumptions, we have
\[
\begin{aligned}
0
&\leq
\frac{n-1}{\ell^2}\int_0^\ell u^\alpha\,\mathrm{d}t
-u(\sigma(0))^\alpha
\left(
H_{\partial M}+\gamma u^{-1}u_\nu
\right)(\sigma(0)) \\
&\leq
\frac{n-1}{\ell^2}\int_0^\ell u^\alpha\,\mathrm{d}t
-(n-1)k\,u(\sigma(0))^\alpha.
\end{aligned}
\]
Since $\alpha\geq 0$, we have
\[
\int_0^\ell u^\alpha\mathrm{d}t\leq \ell(\sup_M u)^\alpha
\]
and
\[
u(\sigma(0))^\alpha\geq (\inf_M u)^\alpha.
\]
Therefore
\[
0
\leq
\frac{n-1}{\ell}(\sup_M u)^\alpha
-(n-1)k(\inf_M u)^\alpha.
\]
Equivalently,
\[
\ell
\leq
\frac1k
\left(
\frac{\sup_M u}{\inf_M u}
\right)^\alpha.
\]
Since $d_g(x,\partial M)\leq \ell$, we obtain
\[
d_g(x,\partial M)
\leq
\frac1k
\left(
\frac{\sup_M u}{\inf_M u}
\right)^{\frac{n-3}{n-1}\gamma}.
\]
Taking the supremum over $x\in M$ gives the desired estimate. The last part of the statement is straightforward.
\end{proof}
\begin{remark}
    In \cref{prop:WeightedInradius}, when $n=2$, the same proof gives $\mathrm{InRad}(M)\leq \frac{1}{k}\left(\frac{\sup_M u}{\inf_M u}\right)^\gamma$, for every $\gamma\geq 0$. In addition, when $n\geq 2$ and $0\leq \gamma < \frac{4}{n-1}$ one can prove
    \[
    \mathrm{InRad}(M)\leq C(n,k,\gamma),
    \]
    for an explicitly computable constant $C(n,k,\gamma)$ only depending on $n,k,\gamma$. The proof is by now standard: one substitutes $\varphi:=u^{-\gamma/2}\psi$ in the analogue of \eqref{eqn:second varitationsummed}, and argues as in the proof of \cref{prop:WeightedInradius}. We omit the details.
\end{remark}

\begin{remark}\label{rem:AX24}
The fixed-endpoint version of Lemma~\ref{lem:SecondVariationGeneral} gives a
slightly better oscillation dependence in the Bonnet--Myers type estimate than
the one obtained in \cite[Theorem 1(1)]{AX24}. Indeed, assume that $M$ is complete,
that
\[
0<\inf_M u\leq \sup_M u<\infty,
\]
and that, for some $K>0$,
\begin{equation}\label{eqn:InequalityRicci}
-\gamma\Delta u + u \mathrm{Ric}\geq (n-1)Ku.
\end{equation}
Let $n\geq 3$, $0\leq \gamma\leq \frac{n-1}{n-2}$, and
\[
\alpha:=\frac{n-3}{n-1}\gamma .
\]
If $p,q\in M$, let $\sigma:[0,\ell]\to M$ be an $E$-minimizer between
$p$ and $q$, where
\[
E(\beta)=\int_\beta u^\gamma\mathrm{d}\ell_g,
\]
and parameterize $\sigma$ by $g$-arclength. Choosing as variation function
\[
\psi(t)=\sin\left(\frac{\pi t}{\ell}\right),
\]
and using the second variation computations in \cref{sec:Secondvariation}, \eqref{eqn:InequalityRicci}, and $0\leq \gamma\leq \frac{n-1}{n-2}$, we obtain
\[
0\leq
(n-1)\int_0^\ell u^\alpha
\left[
\frac{\pi^2}{\ell^2}
\cos^2\left(\frac{\pi t}{\ell}\right)
-
K
\sin^2\left(\frac{\pi t}{\ell}\right)
\right]\mathrm{d}t .
\]
Estimating $u^\alpha$ by $(\sup_M u)^\alpha$ in the positive term and by
$(\inf_M u)^\alpha$ in the negative term gives
\[
0\leq
(n-1)\left[
(\sup_M u)^\alpha\frac{\pi^2}{2\ell}
-
K(\inf_M u)^\alpha\frac{\ell}{2}
\right].
\]
Hence
\[
\ell
\leq
\frac{\pi}{\sqrt K}
\left(
\frac{\sup_M u}{\inf_M u}
\right)^{\alpha/2}.
\]
Since $d_g(p,q)\leq \ell$, taking the supremum over $p,q\in M$ yields
\[
\diam(M,g)
\leq
\frac{\pi}{\sqrt K}
\left(
\frac{\sup_M u}{\inf_M u}
\right)^{\frac{n-3}{2(n-1)}\gamma}.
\]
Thus the fixed-endpoint second-variation argument improves the oscillation
exponent in the corresponding diameter estimate from
\[
\frac{n-3}{n-1}\gamma,
\]
obtained in \cite[Theorem 1(1)]{AX24} using $\mu$-bubble techniques, to
\[
\frac{n-3}{2(n-1)}\gamma .
\]
\end{remark}
\printbibliography

\end{document}